\documentclass[a4paper]{amsart}

\usepackage{amsmath,amsthm,amssymb,enumerate}
\usepackage{epsfig}
\usepackage{amssymb}
\usepackage{amsmath}
\usepackage{amssymb}
\usepackage{amsmath,amsthm}
\usepackage[latin1]{inputenc}
\usepackage[T1]{fontenc}
\usepackage{path}
\usepackage{ae,aecompl}
\usepackage{amsfonts}
\usepackage{amsxtra}
\usepackage{euscript,mathrsfs}
\usepackage{color}
\usepackage[left=3.8cm,right=3.8cm,top=4cm,bottom=4cm]{geometry}
\usepackage[citecolor=blue,colorlinks=true]{hyperref}
\allowdisplaybreaks

\usepackage{enumitem}
\setenumerate{label={\rm (\alph{*})}}

\usepackage{amsfonts}
\usepackage{amsxtra}

\numberwithin{equation}{section}

  \newtheorem{theorem}{Theorem}[section]
  \newtheorem{lemma}[theorem]{Lemma}
  \newtheorem{proposition}[theorem]{Proposition}
  \newtheorem{corollary}[theorem]{Corollary}

  \newtheorem{definition}[theorem]{Definition}

  \newtheorem{remark}[theorem]{Remark}

\newcommand{\bfw}{\mathbf{w}}
\newcommand{\bfu}{\mathbf{u}}
\newcommand{\bfv}{\mathbf{v}}
\newcommand{\bfq}{\mathbf{q}}
\newcommand{\bfg}{\mathbf{g}}
\newcommand{\bfh}{\mathbf{h}}
\newcommand{\bfr}{\mathbf{r}}

\newcommand{\bfF}{\mathbf{F}}

\newcommand{\bfphi}{\boldsymbol{\varphi}}
\newcommand{\bfvarphi}{\boldsymbol{\varphi}}

\newcommand{\diver}{\divergence}

\newcommand{\R}{\mathbb R}
\newcommand{\N}{\mathbb N}

\newcommand{\E}{\mathbb E}
\newcommand{\p}{\mathbb P}

\newcommand{\dd}{\mathrm{d}}
\newcommand{\dx}{\,\mathrm{d}x}
\newcommand{\dt}{\,\mathrm{d}t}
\newcommand{\dxt}{\,\mathrm{d}x\,\mathrm{d}t}

\newcommand{\dxs}{\,\mathrm{d}x\,\mathrm{d}\sigma}

\newcommand{\dif}{\mathrm{d}}

\newcommand{\mf}{\mathscr{F}}
\newcommand{\mr}{\mathbb{R}}
\newcommand{\prst}{\mathbb{P}}
\newcommand{\stred}{\mathbb{E}}
\newcommand{\ind}{\mathbf{1}}
\newcommand{\mn}{\mathbb{N}}

\newcommand{\mt}{\mathbb{T}^N}

\DeclareMathOperator{\supp}{supp}

\DeclareMathOperator{\divergence}{div}

\newcommand{\distr}{\overset{d}{\sim}}
\newcommand{\tec}{{\overset{\cdot}{}}}

\newcommand{\bu}{\mathbf u}
\newcommand{\bq}{\mathbf q}
\newcommand{\rmk}[1]{\textcolor{red}{#1}}

\newcommand{\vr}{\varrho}

\newcommand{\vu}{\vc{u}}
\newcommand{\vc}[1]{{\bf #1}}

\newcommand{\intTor}[1]{\int_{\tor} #1 \ \dx}

\newcommand{\tor}{\mathbb{T}^N}

\allowdisplaybreaks

\begin{document}

%\begin{frontmatter}

\title{Incompressible limit for compressible fluids with stochastic forcing}

\author{Dominic Breit}
\address[D. Breit]{Department of Mathematics, Heriot-Watt University, Riccarton Edinburgh EH14 4AS, UK}
\email{d.breit@hw.ac.uk}

\author{Eduard Feireisl}
\address[E. Feireisl]{Institute of Mathematics AS CR
\v{Z}itn\'a 25
CZ - 115 67 Praha 1
Czech Republic}
\email{feireisl@math.cas.cz}

\author{Martina Hofmanov\'a}
\address[M. Hofmanov\'a]{Technical University Berlin, Institute of Mathematics, Stra\ss e des 17. Juni 136, 10623 Berlin, Germany}
\email{hofmanov@math.tu-berlin.de}

\begin{abstract}

We study the asymptotic behavior of the isentropic Navier-Stokes system driven by a multiplicative stochastic forcing
in the compressible regime, where the Mach number approaches zero. Our approach is based on the recently developed concept
of weak martingale solution to the primitive system, uniform bounds derived from a stochastic analogue of the modulated energy inequality, and careful analysis of acoustic waves. A stochastic incompressible Navier-Stokes system is identified as the limit problem.

\end{abstract}

\subjclass[2010]{60H15, 35R60, 76N10,  35Q30}
\keywords{Compressible fluids, stochastic Navier-Stokes equations, incompressible limit, weak solution, martingale solution}

\date{\today}

\maketitle

%\end{frontmatter}

\section{Introduction}

Singular limit processes bridge the gap between fluid motion considered in different geometries, times scales, and/or under different constitutive relations as the case may be. In their pioneering paper, Klainerman and Majda \cite{KlMa} proposed a general approach to these problems in the context of hyperbolic
conservation laws, in particular, they examine the passage from compressible to incompressible fluid flow motion via the low Mach number limit. As the problems are typically non-linear, the method applies in general only on short time intervals on which regular solutions are known to exist.
A qualitatively new way, at least in the framework of \mbox{viscous} fluids, has been open by the mathematical theory of weak solutions
developed by P.-L. Lions \cite{Li2}. In a series of papers, Lions and Masmoudi \cite{LiMa}, \cite{LiMa2} (see also Desjardins, Grenier \cite{DeGr},
Desjardins et al. \cite{DeGrLiMa}) studied various singular limits for the \emph{barotropic} Navier-Stokes system, among which the incompressible (low Mach number)
limit. The incompressible limit is characterized with a large speed of the acoustic waves becoming infinite in the asymptotic regime. Accordingly,
the fluid density approaches a constant and the velocity solenoidal. The limit behavior is described by the standard \emph{incompressible} Navier
-Stokes system.

In the present paper, we study the \emph{compressible--incompressible} scenario in the context of stochastically driven fluids. Specifically,
we consider the Navier--Stokes system for an isentropic compressible viscous fluid driven by a multiplicative stochastic forcing and study the asymptotic behavior of solutions in the low Mach number regime. To avoid the well known difficulties
due to the presence of a boundary layer in the case of no-slip boundary conditiones (cf. Desjardins et al. \cite{DeGrLiMa}), we restrict
ourselves to the motion in the ``flat'' $N$-dimensional torus $\mt=\left( [0,2\pi]|_{ \{ 0 , 2\pi\} } \right)^N$, $N=2,3$ and on a finite
time interval $(0,T)$; we set $Q=(0,T)\times\mt$. We study the limit as $\varepsilon\rightarrow0$ in the following system which governs the time evolution of the density $\varrho$ and the velocity $\bu$ of a compressible viscous fluid:
\begin{subequations}\label{eq:}
%\begin{equation}
 \begin{align}
  \dif \varrho+\divergence(\varrho\bu)\dif t&=0,\label{eq1}\\%&\qquad\text{ in } &Q,\label{eq1}\\
  \dif(\varrho\bu)+\big[\divergence(\varrho\bu\otimes\bu)-\nu\Delta\bu-(\lambda+\nu)\nabla\divergence\bfu+\frac{1}{\varepsilon^2}\nabla p(\varrho)\big]\dif t&=\varPhi(\varrho,\varrho\bu) \,\dif W.\label{eq2}%&\qquad\text{ in } &Q,\label{eq2}\\
 % \varrho(0)=\varrho_0,\qquad (\varrho\bu)(0)&=\bq_0\label{eq3}%&\qquad\text{ in }&\mt.\label{eq3}
 \end{align}
%\end{equation}
\end{subequations}
Here $p(\varrho)$ is the pressure which is supposed to follow the $\gamma$-law, i.e. $p(\varrho)=\varrho^\gamma$ where $\gamma>N/2$; the viscosity coefficients $\nu,\,\lambda$ satisfy
$$\nu>0,\quad\lambda+\frac{2}{3}\nu\geq0.$$
The driving process $W$ is a cylindrical Wiener process defined on some probability space $(\Omega,\mf,\prst)$ and the coefficient $\varPhi$ is a linear function of momentum $\varrho\bfu$ and a generally nonlinear function of density $\varrho$ satisfying suitable growth conditions. The precise description of the problem setting will be given in the next section.\\
The parameter $\varepsilon$ in \eqref{eq2} is proportional to the \emph{Mach number} (the ratio of the characteristic flow velocity and
the speed of sound). From a physical point of view the fluid should behave (asymptotically) like an incompressible one if the density is close to a constant, the velocity is small and we look at large time scales. A suitable scaling of the Navier-Stokes system results in \eqref{eq2} with a small parameter $\varepsilon$, see Klein et al.
\cite{KBSMRMHS}. 
In the limit of (\ref{eq1}--\ref{eq2}) we recover the stochastic Navier--Stokes system for incompressible fluids, that is,
\begin{subequations}\label{eq:lim}
%\begin{equation}
 \begin{align}
  \dif\bu+\big[\divergence(\bu\otimes\bu)-\nu\Delta\bu+\nabla \pi\big]\dif t&=\varPsi(\bu) \,\dif W,\label{eq2lim}\\
    \divergence(\bu)&=0,\label{eq1lim}
 % \varrho(0)=\varrho_0,\qquad (\varrho\bu)(0)&=\bq_0\label{eq3}%&\qquad\text{ in }&\mt.\label{eq3}
 \end{align}
%\end{equation}
\end{subequations}
where $\pi$ denotes the associated pressure and $\varPsi(\bfu)= \mathcal{P}_H \varPhi (1,\bfu)$, with
$\mathcal P_H$ being the Helmholtz projection onto the space of solenoidal vector fields.
To be more precise, we show that for a given initial law $\Lambda$ for \eqref{eq:} and the \emph{ill--prepared} initial data for the compressible Navier--Stokes system \eqref{eq:}, the approximate densities converge to a constant whereas the velocities converge \emph{in law} to a weak martingale solution to the incompressible Navie-r-Stokes system \eqref{eq:lim} with the initial law $\Lambda$. This result is then strengthened in dimension two where we are able to prove the almost sure convergence of the velocities.

Our approach is based on the concept of \emph{finite energy weak martingale solution} to the compressible Navier--Stokes system (\ref{eq:}), whose existence was established recently in
\cite{BrHo} and extends the approach in \cite{feireisl1} to the stochastic setting, see Section \ref{sec:framework} for more details. Similarly to its deterministic counterpart, the low Mach number limit problem features two
essential difficulties:
\begin{itemize}
\item
finding suitable uniform bounds independent of the scaling parameter $\epsilon$;
\item
analysis of rapidly oscillating \emph{acoustic waves}, at least in the case of ill-prepared data.
\end{itemize}
Here, the necessary uniform bounds follow directly from the associated stochastic analogue of the energy inequality exploiting the basic
properties of It\^o's integral, see Section \ref{UB}.
The propagation of acoustic waves is described by a stochastic variant of Lighthill's acoustic analogy: A linear wave equation driven by a stochastic forcing, see Section \ref{AW}. The desired estimates are obtained via the deterministic approach, specifically the
 so-called local method proposed by Lions and Masmoudi
\cite{LiMa,LiMa2}, adapted to the stochastic setting.

A significant difference in comparison to the deterministic situation is the corresponding compactness argument. In general it is not possible to get any compactness in $\omega$ as no topological structure on
the sample space $\Omega$ is assumed. To overcome this difficulty, it is classical to rather concentrate
on compactness of the set of laws of the approximations and apply the Skorokhod representation
theorem. It gives existence of a new probability space with a sequence of random variables
that have the same laws as the original ones and that in addition converge almost surely.
However, the Skorokhod representation Theorem is restricted to
metric spaces but the structure of the compressible Navier--Stokes equations naturally leads
to weakly converging sequences. On account of this we work with the Jakubowski--Skorokhod
Theorem which is valid on a large class of topological spaces (including separable Banach
spaces with weak topology). In the two-dimensional case we gain a stronger convergence result (see Theorem \ref{thm:main2d}). This is based on the uniqueness for the system \eqref{eq:lim} and a new version of the Gy\"{o}ngy--Krylov characterization of convergence in probability \cite{krylov} which applies to the setting of quasi-Polish spaces (see Proposition \ref{diagonal}).

We point out that the gradient part of the velocity converges only weakly to zero due to the presence of the acoustic waves, and, consequently,
the limit in the stochastic forcing $\Phi(\varrho, \varrho \bu) {\rm d}W$ can be performed only if $\Phi$ is \emph{linear} with respect to $\varrho \bu$. However, this setting already covers the particular case of
\begin{align*}
\Phi(\varrho,\varrho\bfu)\,\dd W= \varrho \,\Phi_1\,\dd W^1+\varrho\bfu\,\Phi_2\,\dd W^2
\end{align*}
with two independent cylindrical Wiener processes $W^1$ and $W^2$ and suitable Hilbert--Schmidt operators $\Phi_1$ and $\Phi_2$, which is the main example we have in mind. Here the first term describes some external force whereas the second one
may be interpreted as a friction force of Brinkman's type, see e.g. Angot et al. \cite{AnBrFa}.

In the case of $\Phi(\varrho,\varrho\bfu)=\varrho\,\Phi_1$, a semi-deterministic approach towards existence for \eqref{eq:} was developed in \cite{feireisl2} (see also \cite{To} for the two-dimensional case). More precisely, this particular case of multiplicative noise
permits reduction of the problem that can be solved pathwise using deterministic arguments only. Nevertheless, it seems that such a pathwise approach is not convenient for the incompressible limit. In particular, uncontrolled quantities appear in the basic energy estimate and therefore the uniform bounds with respect to the parameter $\varepsilon$ are lost. On the contrary, the stochastic method of the present paper heavily depends on the martingale properties of the It\^o's stochastic integral which gives sufficient control of the expected values of all the necessary quantities.

We point out that a noise depending on the velocity $u$ in a \emph{non-linear} way cannot be unfortunately handled by the present method. This is due to only weak convergence of the velocity due to the oscillations generated by acoustic waves - a problem occurring already at the deterministic level, cf. Lions and Masmoudi \cite{LiMa}

The exposition is organized as follows. In Section \ref{sec:framework} we continue with the introductory part: we introduce the basic set-up, the concept of solution and state the main results in Theorem \ref{thm:main} and Theorem \ref{thm:main2d}. The remainder of the paper is then devoted to its proof.

\section{Mathematical framework and the main result}
\label{sec:framework}

Throughout the whole text, the symbols $W^{l,p}$ will denote the Sobolov space of functions having distributional derivatives up to order $l$ integrable in $L^p$. We will also use $W^{l,2}(\mt)$ for $l \in \R$ to denote the space of distributions $v$ defined on $\mt$ with the finite norm
\begin{equation}\label{trigo}
\sum_{k \in \mathbb{Z}} k^{2l} |c_k(v)|^2 < \infty,
\end{equation}
where $c_k$ denote the Fourier coefficients with respect to the standard trigonometric basis $\{ \exp(ikx) \}_{k \in\mathbb{Z}}$.

To begin with, let us set up the precise conditions on the random perturbation of the system \eqref{eq:}. Let $(\Omega,\mf,(\mf_t)_{t\geq0},\prst)$ be a stochastic basis with a complete, right-continuous filtration. The process $W$ is a cylindrical Wiener process, that is, $W(t)=\sum_{k\geq1}\beta_k(t) e_k$ with $(\beta_k)_{k\geq1}$ being mutually independent real-valued standard Wiener processes relative to $(\mf_t)_{t\geq0}$ and $(e_k)_{k\geq1}$ a complete orthonormal system in a sepa\-rable Hilbert space $\mathfrak{U}$.
To give the precise definition of the diffusion coefficient $\varPhi$, consider $\rho\in L^\gamma(\mt)$, $\rho\geq0$, and $\bfv\in L^2(\mt)$ such that $\sqrt\rho\bfv\in L^2(\mt)$. Denote $\bfq=\rho\bfv$ and let $\,\varPhi(\rho,\bq):\mathfrak{U}\rightarrow L^1(\mt)$ be defined as follows%\marginpar{\textcolor{red}{$\bfH_k$ has to be a constant??}\rmk{because we only have convergence of $\mathcal P (\varrho_\varepsilon\bfu_\varepsilon)$ and not the gradient part.}}
$$\varPhi(\rho,\bq)e_k=\bfg_k(\cdot,\rho(\cdot),\bq(\cdot))=\bfh_k(\cdot,\rho(\cdot))+\alpha_k  \bq(\cdot) ,$$
where the coefficients $\alpha_k\in\mr$ are constants and $\bfh_{k}:\mt\times\mr \to \mr $ are $C^1$-functions that satisfy
\begin{align}
\sum_{k \geq 1} | \alpha_k |^2& < \infty,\label{growth1-}\\
\sum_{k\geq 1}|\bfh_{k}(x,\rho)|^2&\leq C(\rho^2+|\rho|^{\gamma+1}),\label{growth1}\\
\quad\sum_{k\geq 1}|\nabla_{\rho} \bfh_{k}(x,\rho)|^2&\leq C(1+|\rho|^{\gamma-1}).\label{growth2}
%\sum_{k\geq 1}\| H_{2,k}\|_{L_{x,\rho}^\infty}^2&\leq C,\label{growth3}\\
%\sum_{k\geq 1}\| \partial_\rho H_{2,k}\|_{L_{x,\rho}^\infty}^2&\leq C.\label{growth4}
\end{align}
Remark that in this setting $L^1(\mt)$ is the natural space for values of the operator $\varPhi(\rho,\rho\bfv)$. Indeed, due to lack of a priori estimates for \eqref{eq:} it is not possible to consider $\varPhi(\rho,\rho\bfv) $ as a mapping with values in a space with higher integrability. This fact brings difficulties concerning the definition of the stochastic integral in \eqref{eq:} because the space $L^1(\mt)$ does not belong among 2-smooth Banach spaces nor among UMD Banach spaces where the theory of stochastic It\^o integration is well-established (see e.g. \cite{b2}, \cite{ondrejat3}, \cite{veraar}). However, since we expect the momentum equation \eqref{eq2} to be satisfied only in the sense of distributions anyway, we make use of the embedding $L^1(\mt)\hookrightarrow W^{-l,2}(\mt)$, which is true provided $l>\frac{N}{2}$, and understand the stochastic integral as a process in the Hilbert space $W^{-l,2}(\mt)$. To be more precise, it is easy to check that under the above assumptions on $\rho$ and $\bfv$, the mapping $\varPhi(\rho,\rho\bfv)$ belongs to $L_2(\mathfrak{U};W^{-l,2}(\mt))$, the space of Hilbert-Schmidt operators from $\mathfrak{U}$ to $W^{-l,2}(\mt)$. Indeed, due to \eqref{growth1-} and \eqref{growth1}
\begin{align}\label{stochest}
\begin{aligned}
\big\|\varPhi(\rho,\rho\bfv)\big\|^2_{L_2(\mathfrak{U};W^{-l,2}_x)}&=\sum_{k\geq1}\|\bfg_k(\rho,\rho\bfv)\|_{W^{-l,2}_x}^2\leq C\sum_{k\geq1}\|\bfg_k(\rho,\rho\bfv)\|_{L^1_x}^2\\
&\leq \sum_{k \geq 1} \left( \int_{\mt} \big(|\bfh_k (x, \rho)| +\rho |\alpha_k  \bfv| \big) \,\dif x \right)^2 \\
%&=C\sum_{k\geq1}\bigg(\int_{\mt}\big|\rho\, g_{k}(x,\rho)+H_{2,k}(x,%\rho)\rho\bfv\big|\,\dif x\bigg)^2\\
&\leq C(\rho)_{\mt}\int_{\mt}\bigg(\sum_{k\geq1}\rho^{-1}|\bfh_{k}(x,\rho)|^2+\sum_{k\geq 1}\rho|\alpha_k\bfv|^2  \bigg)\dif x \\
&\leq C(\rho)_{\mt}\int_{\mt}\big(\rho+\rho^\gamma+\rho|\bfv|^2\big)\,\dif x<\infty,
\end{aligned}
\end{align}
where $(\rho)_{\mt}$ denotes the mean value of $\rho$ over $\mt$.
Consequently, if
\begin{align*}
\rho&\in L^\gamma(\Omega\times(0,T),\mathcal{P},\dif\prst\otimes\dif t;L^\gamma(\mt)),\\
\sqrt\rho\bfv&\in L^2(\Omega\times(0,T),\mathcal{P},\dif\prst\otimes\dif t;L^2(\mt)),
\end{align*}
where $\mathcal{P}$ denotes the progressively measurable $\sigma$-algebra associated to $(\mf_t)$, and the mean value $(\rho(t))_{\mt}$ (that is constant in $t$ but in general depends on $\omega$) is for instance essentially bounded
%$$\big(\rho(\omega,t)\big)_{\mt}=\int_{\mt}\rho(\omega,t,x)\,\dif x\leq M\in(0,\infty),$$
then the stochastic integral $\int_0^\tec\varPhi(\rho,\rho\bfv)\,\dif W$ is a well-defined $(\mf_t)$-martingale taking values in $W^{-l,2}(\mt)$.
Finally, we define the auxiliary space $\mathfrak{U}_0\supset\mathfrak{U}$ via
$$\mathfrak{U}_0=\bigg\{v=\sum_{k\geq1}c_k e_k;\;\sum_{k\geq1}\frac{c_k^2}{k^2}<\infty\bigg\},$$
endowed with the norm
$$\|v\|^2_{\mathfrak{U}_0}=\sum_{k\geq1}\frac{c_k^2}{k^2},\qquad v=\sum_{k\geq1}c_k e_k.$$
Note that the embedding $\mathfrak{U}\hookrightarrow\mathfrak{U}_0$ is Hilbert-Schmidt. Moreover, trajectories of $W$ are $\prst$-a.s. in $C([0,T];\mathfrak{U}_0)$ (see \cite{daprato}).

\subsection{The concept of solution and the main result}
\label{subsec:solution}

Existence of the so-called finite energy weak martingale solution to the stochastic Navier-Stokes system for compressible fluids, in particular \eqref{eq:}, was recently established in \cite{BrHo}. Let us recall the corresponding definition of a solution and the existence result.
\begin{definition}
\label{def:sol} A quantity
\[
\left[ (\Omega,\mf,(\mf_t),\prst); \vr, \vu, W \right]
\]
is called a weak martingale solution to problem (\ref{eq1}--\ref{eq2}) with the initial law $\Lambda$ provided:
\begin{enumerate}
\item
$(\Omega,\mf,(\mf_t),\prst)$ is a stochastic basis with a complete right-continuous filtration;
\item $W$ is an $(\mf_t)$-cylindrical Wiener process;
\item the density $\vr$ satisfies $\vr \geq 0$, $t \mapsto \left< \vr(t, \cdot), \psi \right> \in C[0,T]$ for any
$\psi \in C^\infty(\tor)$
$\mathbb{P}-$a.s., the function $t \mapsto \left< \vr(t, \cdot), \psi \right>$
is progressively measurable,
and
\[
\E\bigg[\sup_{t \in [0,T]} \| \vr(t,\cdot) \|^p_{L^\gamma(\tor)} \bigg] < \infty \ \mbox{for all}\ 1 \leq p < \infty;
\]
\item the velocity field $\vu$ is adapted, $\vu \in L^2(\Omega \times (0,T); W^{1,2}(\tor))$,
\[
\E\bigg[\bigg( \int_0^T \| \vu \|^2_{W^{1,2}(\tor)} \ \dt \bigg)^p \bigg] < \infty\ \mbox{for all}\ 1 \leq p < \infty;
\]
\item the momentum $\vr \vu$ satisfies $t \mapsto \left< \vr \vu, \phi \right> \in C[0,T]$ for any $\phi \in C^\infty(\tor)$
$\mathbb{P}-$a.s., the function $t \mapsto \left< \vr \vu, \phi \right>$ is progressively measurable,
\[
\E \sup_{t \in [0,T]} \left\| \vr \vu \right\|^p_{L^{\frac{2 \gamma}{\gamma + 1}}}  < \infty\ \mbox{for all}\  1 \leq p < \infty;
\]
\item $\Lambda=\mathbb{P}\circ \left( \vr(0), \vr \vu (0) \right)^{-1} $,
\item for all $\psi\in C^\infty(\mt)$ and
 $\bfvarphi\in C^\infty(\mt)$ and all $t\in[0,T]$ it holds $\prst$-a.s.
\begin{align*}
\big\langle\varrho(t),\psi\big\rangle&=\big\langle\varrho(0),\psi\big\rangle+\int_0^t\big\langle\varrho\bfu,\nabla\psi\big\rangle\,\dif s,\\
\big\langle\varrho\bfu(t),\bfvarphi\big\rangle&=\big\langle\varrho \bfu(0),\bfvarphi\big\rangle+\int_0^t\big\langle\varrho\bfu\otimes\bfu,\nabla\bfvarphi\big\rangle\,\dif s-\nu\int_0^t\big\langle\nabla\bfu,\nabla\bfvarphi\big\rangle\,\dif s\\
&\quad-(\lambda+\nu)\int_0^t\big\langle\divergence\bfu,\divergence\bfvarphi\big\rangle\,\dif s+\frac{1}{\varepsilon^2}\int_0^t\big\langle\rho^\gamma,\divergence\bfvarphi\big\rangle\,\dif s\\
&\quad+\int_0^t\big\langle\varPhi(\varrho,\varrho\bfu)\,\dif W,\bfvarphi\big\rangle,
\end{align*}
%\item for all test functions $\psi \in C^\infty(\tor)$, $\phi \in C^\infty(\tor)$ and all $t \in [0,T]$ it holds $\mathbb{P}$-a.s.:
%\begin{align}
%\nonumber
%\dif\left< \vr, \psi \right> &= \left< \vr \vu , \Grad \psi \right> \dt,
%\\
%\nonumber \dif\left< \vr \vu, \phi \right>  &= \Big[ \left< \vr \vu \otimes \vu , \Grad \phi \right>  - \left< \tn{S}(\Grad \vu), \Grad \phi \right>
%+ \left< p(\vr), \Div \phi \right> \Big]\dt
%+ \left< \varPhi (\vr, \vr \vu) , \phi \right> \dif W;
%\end{align}
%\item for all $p\in[1,\infty)$ the following energy inequality holds true
%\begin{equation*}\label{energy}
%\begin{split}
%&\stred\bigg[\sup_{0\leq t\leq T}\int_ {\mt} \Big(\frac{1}{2} \varrho(t)\big| \bfu(t)\big|^2+\frac{1}{\varepsilon^2(\gamma-1)}\varrho^\gamma (t)\Big)\dif x\bigg]^p\\
%&\quad+\stred\bigg[\int_0^{T}\int_ {\mt}\Big(\nu |\nabla \bfu |^2+(\lambda+\nu)|\diver\bfu|^2\Big)\dif x\,\dif s\bigg]^p\\
%& \leq \,C(p)\,\stred\bigg[\int_{\mt} \Big(\frac{1}{2} \frac{ |\varrho \bfu(0)|^2 }{\varrho(0)} +\frac{a}{(\gamma-1)} \varrho(0)^\gamma\Big)\dif x+1\bigg]^p.
%\end{split}
%\end{equation*}
\end{enumerate}
\end{definition}

The present problem requires a refined concept of \emph{finite energy weak solution} similar to that introduced in \cite{BFH}.
The relevant existence result is proved in \cite{BrHo}:
\begin{theorem} \label{thm:exist}
%Let the pressure $p$ be as in (\ref{press}) and let $\vc{G}_k$ be continuously differentiable satisfying (\ref{FG1}), (\ref{FG2}).
%Let the initial law $\Lambda$ be given on the space $L^\gamma (\tor) \times L^{\frac{2 \gamma}{\gamma + 1}}(\tor;R^3)$ and
%\begin{eqnarray}
%\nonumber
%\Lambda \Big\{ (\vr, \vc{q} ) \in L^\gamma (\tor) & \times & L^{\frac{2 \gamma}{\gamma + 1}}(\tor;R^3),\ \vr \geq 0,  \\
%\nonumber  0 < M_1 \leq \intTor{ \vr } \leq M_2,\ \vc{q} &=& 0 \ \mbox{a.e. on the set} \ \{ \vr = 0\} \Big\} = 1 ,
%\end{eqnarray}
%for certain constants $0<M_1<M_2$,
%\[
%\int_{ L^\gamma \times L^{2\gamma/(\gamma + 1)}} \left\| \frac{1}{2} \frac{ |\vc{q}|^2}{\vr}  + H(\vr) \right\|_{L^1(\tor)}^p
%\ {\rm d} \Lambda (\vr, \vc{q} ) \leq c(p) < \infty
%\]
%for any $1 \leq p < \infty$.
Assume that for the initial law $\Lambda$ there exists $M\in(0,\infty)$ such that
\begin{equation*}
\Lambda\Big\{(\rho,\bfq)\in L^\gamma(\mt)\times L^\frac{2\gamma}{\gamma+1}(\mt);\, \rho\geq0,\;(\rho)_{\mt}\leq M,\;\bfq(x)=0\;\text{if}\;\rho(x)=0\Big\}=1,
\end{equation*}
and that for all $p\in[1,\infty)$ the following moment estimate holds true
\begin{equation*}
\int_{L^\gamma_x\times L^\frac{2\gamma}{\gamma+1}_x}\bigg\|\frac{1}{2}\frac{|\bfq|^2}{\rho}+ \rho^\gamma\bigg\|_{L^1_x}^p\,\dif\Lambda(\rho,\bfq)\leq C_\varepsilon.
\end{equation*}
Then the Navier--Stokes system (\ref{eq1}--\ref{eq2}) possesses at least one weak martingale solution with the initial law $\Lambda$.
In addition, the equation of continuity (\ref{eq1}) holds also in the renormalized sense
\[
\dif \left< b(\vr), \psi \right> = \left< b(\vr) \vu, \nabla \psi \right> {\rm d}t - \left<  \left( b(\vr) - b'(\vr) \vr \right) \diver \vu,
\psi \right>\dt
\]
for any test function $\psi \in C^\infty(\tor)$, and any $b \in C^1[0,\infty)$, $b'(\vr) = $ for $\vr \geq \vr_g$.
Moreover, the energy inequality
\begin{align}
\label{energy}
\begin{aligned}
\E\bigg[ &\sup_{t \in [0,T]} \intTor{ \Big[ \frac{|\vr \vu|^2 }{2 \vr} + \frac{1}{\varepsilon^2} H(\vr) \Big] } \bigg]^p
  \\
  &\qquad+ \E\bigg[ \int_0^T \intTor{ \nu|\nabla \vu|^2+(\lambda+\nu)|\diver \bfu|^2} \dt  \bigg]^p \\
 &\leq
c(p,T) \,\E\bigg[ \bigg( \intTor{ \Big[ \frac{|(\vr \vu) (0)|^2 }{2 \vr(0)}+ \frac{1}{\varepsilon^2} H(\vr(0)) \Big] } \bigg)^p +1\bigg]
\end{aligned}
\end{align}
hold
for any $1 \leq p < \infty$, where
\[
H(\vr), \ H''(\vr) = \frac{p'(\vr)}{\vr}
\]
is the so-called pressure potential.
\end{theorem}
\begin{remark}\label{rem:new}
The pressure potential $H$ is determined up to a linear function. In particular, one can take
\[
H(\varrho) =  \frac{1}{\gamma - 1} \left( \varrho^{\gamma} - \gamma \overline{\varrho}^{\gamma - 1}(\varrho - \overline{\varrho}) - \overline{\varrho}^\gamma \right)
\]
for any constant $\overline{\varrho} > 0$.
\end{remark}
\begin{remark}
Because of \eqref{energy} this solution is called finite energy weak martingale solution. The constant $c(p,T)$ depends on $T$ (via Gronwall's lemma) and the constants in \eqref{growth1-}--\eqref{growth2} but is independent of $\varepsilon$.
\end{remark}
\begin{proof}
The proof of \eqref{energy} is somewhat hidden in \cite{BrHo} as multi-layer approximation scheme is needed. For the reader's convenience we give a formal proof on the level of smooth solutions.
In order to obtain a priori estimates we apply It\^{o}'s formula to the functional $f(\bfq, \varrho)=\frac{1}{2}\int_ {\mathbb T^N} \frac{|\bfq|^2}{ \varrho}\dx$. This corresponds exactly to the test with $ \bfu$ in the momentum equation and $\frac{1}{2}|\bfu|^2$ in the continuity equation from the deterministic case. We gain
\begin{align*}
\frac{1}{2}&\int_ {\mathbb T^N}  \varrho| \bfu|^2\dx\\&=\frac{1}{2}\int_ {\mathbb T^N}  \frac{|(\varrho\bfu)(0)|^2}{\varrho(0)}\dx-\nu\int_0^t\int_ {\mathbb T^N} |\nabla \bfu|^2\dxs-(\lambda+\nu)\int_0^t\int_{\mt}|\diver\bfu|^2\dxs\\&+\int_0^t\int_ {\mathbb T^N} \varrho \bfu\otimes \bfu:\nabla \bfu\dxs
+\frac{1}{\varepsilon^2}\int_0^t\int_ {\mathbb T^N} \varrho^\gamma\diver \bfu\dxs-\frac{1}{2}\int_0^t\int_ {\mathbb T^N}| \bfu|^2\,\dd \varrho\\&+\int_0^t\int_ {\mathbb T^N} \bfu\cdot\Phi(\varrho, \varrho \bfu)\dx\,\dd W
+\frac{1}{2}\int_0^t \varrho^{-1}\,\dd\Big\langle\int_0^{\cdot}\Phi(\varrho,  \varrho \bfu)\,\dd W\Big\rangle.
\end{align*}
In the following we use the renormalized equation
\begin{equation}\label{4116}
\int_0^t\int_ {\mathbb T^N} \varrho^\gamma \diver \bfu\dxs=- \int_ {\mathbb T^N} H(\varrho) \dx + \int_ {\mathbb T^N} H(\varrho(0)) \dx
\end{equation}
In fact, \eqref{4116} is a consequence of the mass conservation
\begin{align}\label{eq:mass}
\int_{\mt}\varrho(t)\dx=\int_{\mt}\varrho(0)\dx\quad\forall t\in[0,T]
\end{align}
which holds due \eqref{eq1}.
Using \eqref{4116} we gain
\begin{align*}
\frac{1}{2}\int_ {\mathbb T^N} & \varrho| \bfu|^2\dx+\nu_0\int_0^t\int_ {\mathbb T^N} |\nabla \bfu|^2\dxs+\frac{1}{\varepsilon^2}\int_ {\mathbb T^3} H(\varrho) \dx\\
&\leq\frac{1}{2}\int_ {\mathbb T^N}  \frac{|(\varrho\bfu)(0)|^2}{\varrho(0)}\dx+\frac{1}{\varepsilon^2}\int_ {\mathbb T^N} H(\varrho(0))\dx\\
&+\int_0^t\int_ {\mathbb T^N} \bfu\cdot\Phi(\varrho,  \varrho \bfu)\dx\,\dd W
+\frac{1}{2}\int_0^t \varrho^{-1}\,\dd\Big\langle\int_0^{\cdot}\Phi(\varrho,  \varrho\bfu)\,\dd W\Big\rangle\\
&=:\frac{1}{2}\int_ {\mathbb T^N}  \frac{|(\varrho\bfu)(0)|^2}{\varrho(0)}\dx+\frac{1}{\varepsilon^2}\int_ {\mathbb T^N} H(\varrho(0))\dx+T_1(t)+T_2(t).
\end{align*}
We apply the $q$-th power on both sides and then take the expectation. Due to \eqref{growth1} and \eqref{eq:mass} we have
\begin{align}\label{eq:789}
\begin{aligned}
T_2(t)&\leq \frac{1}{2}\sum_k \int_0^t\int_ {\mathbb T^N} \varrho^{-1} |\bfg_k( \varrho, \varrho \bfu)|^2\dxt
\leq \,c\,\int_0^t\int_ {\mathbb T^N} \big(\varrho | \bfu|^2+\varrho^{\gamma}+1\big)\dxt\\
&\leq \,c\,\int_0^t\int_ {\mathbb T^N} \Big(\varrho | \bfu|^2+\frac{1}{\varepsilon^2} H(\varrho) +1\Big)\dxt.
\end{aligned}
\end{align}
%So we have
%\begin{align*}
%\E\bigg[\sup_{t\in(0,T)}\int_ {\mathbb T^N} & \varrho_\varepsilon| \bfu_\varepsilon|^2\dx+\int_0^T\int_ {\mathbb T^N} |\nabla \bfu_\varepsilon|^2\dxs+\sup_{t\in(0,T)}\int_ {\mathbb T^N} \frac{\varrho^\gamma_\varepsilon}{\gamma-1}\dx\bigg]^q\\
%&\leq \,c\,\E\bigg[\int_{\mathbb T^N}\Big(\frac{1}{2}\frac{\bfq_0^2}{\rho_0^2}+\frac{\varrho_0^ \gamma}{\gamma-1}\Big)\dx+\sup_{t\in(0,T)}|T_2(t)|^q\bigg].
%\end{align*}
As a consequence of Burgholder-Davis-Gundy inequality, \eqref{growth1} and \eqref{eq:mass} we gain for $p\geq1$
\begin{align*}
\E\bigg[\sup_{t\in(0,T)}|T_2(t)|\bigg]^{p}&=\E\bigg[\sup_{t\in(0,T)}\Big|\int_0^t\int_ {\mathbb T^N} \bfu\cdot\Phi(\varrho,  \varrho \bfu)\dx\,\dd W_\sigma\Big|\bigg]^{p}\\
&=\E\bigg[\sup_{t\in(0,T)}\Big|\int_0^t\sum_i\int_ {\mathbb T^N} \bfu\cdot \bfg_i(\varrho,  \varrho\bfu)\dx\,\dd\beta_i(\sigma)\Big|\bigg]^{p}\\
&\leq c\,\E\bigg[\int_0^T\sum_i\bigg(\int_ {\mathbb T^N} \bfu\cdot \bfg_i( \varrho, \varrho\bfu)\dx\bigg)^2\dt\bigg]^{\frac{p}{2}}\\
&\leq c\,\E\bigg[\int_0^T\sum_i \bigg(\int_ {\mathbb T^N} \varrho| \bfu|^2\dx\bigg)\bigg(\int_ {\mathbb T^N} \varrho^{-1}| \bfg_i(\varrho, \varrho\bfu)|^2\dx\bigg)\dt\bigg]^{\frac{p}{2}}%\\
%&\leq c\,\E\bigg[\int_0^T\bigg(\int_ {\mathbb T^N} \varrho| \bfu|^2\dx\bigg)^2\dt+\int_0^T\bigg(\int_ {\mathbb T^N} \varrho^\gamma\dx\bigg)^2\dt+1\bigg]^{\frac{p}{2}}\\
%&\leq c\,\E\bigg[\int_0^T\bigg(\int_ {\mathbb T^N} \varrho| \bfu|^2\dx\bigg)^2\dt+\int_0^T\bigg(\int_ {\mathbb T^N} \frac{\varrho^\gamma-\ell(\varrho)}{\varepsilon^2(\gamma-1)}\dx\bigg)^2\dt\bigg]^{\frac{p}{2}}\\
%&+c\,\E\bigg[\int_{\mt}\ell(\varrho(0))+1\bigg]^{\frac{p}{2}}
\end{align*}
and by Young's inequality and a computation similar to \eqref{eq:789} we gain for every $\delta>0$ 
\begin{align*}
\E\bigg[\sup_{t\in(0,T)}|T_2(t)|\bigg]^{p}
&\leq \delta\,\E\bigg[\sup_{t\in(0,T)}\int_ {\mathbb T^N} \varrho| \bfu|^2\dx\bigg]^{p}\\
&\quad+c(\delta)\,\E\bigg[\int_0^T\int_ {\mathbb T^N} \Big(\varrho | \bfu|^2+\frac{1}{\varepsilon^2} H(\varrho) +1\Big)\dx\dt\bigg]^p.
%&+\delta\,\E\bigg[\sup_{t\in(0,T)}\int_ {\mathbb T^N} \frac{(\varrho^\gamma-\ell(\varrho))}{\varepsilon^2(\gamma-1)}\dx\bigg]^p+c(\delta)\,\E\bigg[\int_0^T\int_ {\mathbb T^N}\frac{(\varrho^\gamma-\ell(\varrho))}{\varepsilon^2(\gamma-1)}\dxt\bigg]^{p}.
\end{align*}
Finally, taking $\delta$ small enough and applying
%\begin{align*}
%\E&\bigg[\sup_{t\in(0,T)}\int_ {\mathbb T^N} \varrho| \bfu|^2\dx+\sup_{t\in(0,T)}\int_ {\mathbb T^N} \frac{(\varrho^\gamma-\ell(\varrho))}{\varepsilon^2(\gamma-1)}\dx+\nu_0\int_0^T\int_ {\mathbb T^N} |\nabla \bfu|^2\dxs\bigg]^p\\
%&\leq \,c\,\E\bigg[\int_{\mathbb T^N}\frac{1}{2}\frac{(\varrho\bfu)(0)^2}{\varrho(0)^2}\dx+\int_ {\mathbb T^N} \ell(\varrho(0))\dx+\int_{\mt}\frac{(\varrho(0)^ \gamma-\ell(\varrho(0)))}{\varepsilon^2(\gamma-1)}\dx+1\bigg]^p\\
%&+c\,\E\bigg[\int_0^T\int_ {\mathbb T^N} \varrho| \bfu|^2\dxt+\int_ {\mathbb T^N} \frac{(\varrho^\gamma-\ell(\varrho))}{\varepsilon^2(\gamma-1)}\dxt\bigg]^{p}.
%\end{align*}
Grownwall's lemma, the inequality \eqref{energy} follows.
\end{proof}
\begin{remark}
In Def. \ref{def:sol} (j) the continuity equation is stated in the renormalized sense.
This is part of the existence result in \cite{BrHo} but will not be used in the remainder of the paper.
\end{remark}

%\begin{hypothesis}
%dfs
%\end{hypothesis}

\color{black}

%
%\begin{theorem}\label{thm:exist}
%Assume that for the initial law $\Lambda$ there exists $M\in(0,\infty)$ such that
%\begin{equation*}
%\Lambda\Big\{(\rho,\bfq)\in L^\gamma(\mt)\times L^\frac{2\gamma}{\gamma+1}(\mt);\, \rho\geq0,\;(\rho)_{\mt}\leq M,\;\bfq(x)=0\;\text{if}\;\rho(x)=0\Big\}=1,
%\end{equation*}
%and that for all $p\in[1,\infty)$ the following moment estimate holds true
%\begin{equation*}
%\int_{L^\gamma_x\times L^\frac{2\gamma}{\gamma+1}_x}\bigg\|\frac{1}{2}\frac{|\bfq|^2}{\rho}+\frac{1}{\varepsilon^2(\gamma-1)}\rho^\gamma\bigg\|_{L^1_x}^p\,\dif\Lambda(\rho,\bfq)\leq C_\varepsilon.
%\end{equation*}
%Then there exists a finite energy weak martingale solution to \eqref{eq:} with the initial data $\Lambda$.
%\end{theorem}

Concerning the incompressible Navier-Stokes system \eqref{eq:lim}, several notions of solution are typically considered depending on the space dimension. From the PDE point of view, we restrict ourselves to weak solutions (although more can be proved in dimension two), i.e. \eqref{eq:lim} is satisfied in the sense of distributions. From the probabilistic point of view, we will consider two concepts, namely, pathwise (or strong) solutions and martingale (or
weak) solutions. In the former one the underlying probability space as well as the driving process is fixed in advance while
in the latter case these stochastic elements become part of the solution of the problem. Clearly, existence of a pathwise solution is stronger and implies existence of a martingale solution. Besides, due to classical Yamada-Watanabe-type argument (see e.g. \cite{krylov}, \cite{pr07}), existence of a pathwise solution follows from existence of a martingale solution together with pathwise uniqueness. The difference lies also in the way how the initial condition is posed: for pathwise solutions we are given a random variable $\bfu_0$ whereas for martingale solutions we can only prescribe an initial law $\Lambda$.

Note that due to our assumptions on the operator $\varPhi$, the stochastic perturbations that we obtain in the limit system \eqref{eq:lim} is affine linear function of the velocity and takes the following form
$$\Psi(\bfv)e_k\,\dif\beta_k=\mathcal{P}_H\varPhi(1,\bfv)e_k\,\dif\beta_k=\big(\mathcal{P}_H \bfh_k(1)+\alpha_k\bfv\big)\dif\beta_k.$$
Besides, due to \eqref{growth1-}, \eqref{growth1} it holds true that
\begin{align}\label{eq:psi}
\begin{aligned}
\|\Psi(\bfv)\|_{L_2(\mathfrak U;L^2_x)}^2&\leq C\big(1+\|\bfv\|_{L^2_x}^2\big),\\
\|\Psi(\bfv)-\Psi(\bfw)\|_{L_2(\mathfrak U;L^2_x)}^2&\leq C\|\bfv-\bfw\|_{L^2_x}^2.
\end{aligned}
\end{align}

In dimension three, existence of a strong solution which is closely related to uniqueness is one the celebrated Millenium Prize Problems and remains unsolved. Therefore, we consider weak martingale solutions, see for instance \cite{CaGa} or \cite{FlGa}.

\begin{definition}\label{def:inc}
Let $\Lambda$ be a Borel probability measure on $L^2(\mt)$. Then
$$\big((\Omega,\mf,(\mf_t),\prst),\bfu,W)$$
is called a weak martingale solution to \eqref{eq:lim} with the initial data $\Lambda$ provided
\begin{enumerate}
\item $(\Omega,\mf,(\mf_t),\prst)$ is a stochastic basis with a complete right-continuous filtration,
\item $W$ is an $(\mf_t)$-cylindrical Wiener process,
\item the velocity field $\vu$ is $(\mf_t)$-adapted, $\vu \in C_w([0,T];L^2_{\diver}(\mt))\cap L^2(0,T; W^{1,2}_{\diver}(\tor))$ $\prst$-a.s.\footnote{By $C_w([0,T];L^2_{\diver}(\mt))$ we denote the space of weakly continuous functions with values in the space of divergence-free vector fields in $L^2(\mt)$.} and
\[
\E\bigg[\sup_{(0,T)}\|\bfu\|_{L^2(\mt)}^2\bigg]^p+\E\bigg[\bigg( \int_0^T \| \vu \|^2_{W^{1,2}(\tor)} \ \dt \bigg)^p \bigg] < \infty\ \mbox{for all}\ 1 \leq p < \infty;
\]
%\item the velocity $\bfu$ is $(\mf_t)$-adapted and
%$$\bfu\in L^2(\Omega;L^2(0,T;W_{\text{div}}^{1,2}(\mt)))\cap L^2(\Omega;C_w([0,T];L^2_{\text{div}}(\mt))),$$
%\item $\Lambda=\prst\circ\bfu(0)^{-1}$,
%\item $\varPhi(\varrho,\varrho\bfu)\in L^2(\Omega\times[0,T],\mathcal{P},\dif\prst\otimes\dif t;L_2(\mathfrak{U};W^{-l,2}(\mt)))$ for some $l>\frac{3}{2}$,
%\item for all $p\in[1,\infty)$ the following energy inequality holds true
%\begin{equation}\label{energy}
%\begin{split}
%&\stred\bigg[\sup_{0\leq t\leq T}\int_{\mt}\Big(\frac{1}{2}\varrho|\bu|^2+\frac{a}{\varepsilon^2(\gamma-1)}\varrho^\gamma\Big)\,\dif x+\int_0^T\int_{\mt}\nu|\nabla\bu|^2+(\lambda+\nu)|\divergence\bu|^2\,\dif x\,\dif s\bigg]^p\leq C_\varepsilon,
%\end{split}
%\end{equation}
\item for all
 $\bfvarphi\in C_{\text{div}}^\infty(\mt)$ and all $t\in[0,T]$ it holds $\prst$-a.s.
\begin{align*}
\big\langle\bfu(t),\bfvarphi\big\rangle&=\big\langle\bfu(0),\bfvarphi\big\rangle+\int_0^t\big\langle\bfu\otimes\bfu,\nabla\bfvarphi\big\rangle\,\dif s-\nu\int_0^t\big\langle\nabla\bfu,\nabla\bfvarphi\big\rangle\,\dif s+\int_0^t\big\langle\varPsi(\bfu)\,\dif W,\bfvarphi\big\rangle.
\end{align*}
%\item Let $b\in C^1(\R)$ such that $b'(z)=0$ for all $z\geq M_b$. Then
%for all $\psi\in C^\infty(\mt)$ and all $t\in[0,T]$ it holds $\prst$-a.s.
%\begin{align*}
%\big\langle b(\varrho(t)),\psi\big\rangle&=\big\langle b(\varrho(0)),\psi\big\rangle+\int_0^t\big\langle b(\varrho)\bfu,\nabla\psi\big\rangle\,\dif s-
%\int_0^t\big\langle \big(b'(\varrho)\varrho-b(\varrho)\bfu)\big)\divergence\bfu,\psi\big\rangle\,\dif s.
%\end{align*}
\end{enumerate}
\end{definition}

Here and hereafter, the substrict ${\rm div}$ refers to the space of solenoidal (divergenceless) functions.

Under the condition \eqref{eq:psi}, the following existence result holds true and can be found for instance in \cite{CaGa} and \cite{FlGa}.

\begin{theorem}\label{thm:inc}
Let $\Lambda$ be a Borel probability measure on $L^2(\mt)$ such that for all $p\in[1,\infty)$
$$\int_{L^2_x}\|\bfv\|_{L^2_x}^p\dif\Lambda(\bfv)\leq C(p).$$
Then there exists a weak martingale solution to \eqref{eq:lim} with initial law $\Gamma$.
\end{theorem}

In dimension two, pathwise uniqueness for weak solutions is known under \eqref{eq:psi}, we refer the reader for instance to \cite{CaCu}, \cite{Ca}. Consequently, we may work with the definition of a weak pathwise solution.

\begin{definition}\label{def:inc2d}
Let $(\Omega,\mf,(\mf_t),\prst)$ be a given stochastic basis with an $(\mf_t)$-cylindrical Wiener process $W$ and let $\bfu_0$ be an $\mf_0$-measurable random variable. Then $\bfu$ is called a weak pathwise solution to \eqref{eq:lim} with the initial condition $\bfu_0$ provided
\begin{enumerate}
\item the velocity field $\vu$ is $(\mf_t)$-adapted, $\vu \in C_w([0,T];L^2_{\diver}(\mt))\cap L^2(0,T; W^{1,2}_{\diver}(\tor))$ $\prst$-a.s. and
\begin{align*}
\E\bigg[\sup_{(0,T)}\|\bfu\|^2_{L^2(\tor)}\bigg]^p+\E\bigg[\bigg( \int_0^T \| \vu \|^2_{W^{1,2}(\tor)} \ \dt \bigg)^p \bigg] < \infty\ \mbox{for all}\ 1 \leq p < \infty;
\end{align*}
%\item the velocity $\bfu$ is $(\mf_t)$-adapted and
%$$\bfu\in L^2(\Omega;L^2(0,T;W_{\text{div}}^{1,2}(\mt)))\cap L^2(\Omega;C_w([0,T];L^2_{\text{div}}(\mt))),$$
\item $\bfu(0)=\bfu_0$ $\p$-a.s.,
\item for all
 $\bfvarphi\in C_{\text{div}}^\infty(\mt)$ and all $t\in[0,T]$ it holds $\prst$-a.s.
\begin{align*}
\big\langle\bfu(t),\bfvarphi\big\rangle&=\big\langle\bfu_0,\bfvarphi\big\rangle+\int_0^t\big\langle\bfu\otimes\bfu,\nabla\bfvarphi\big\rangle\,\dif s-\nu\int_0^t\big\langle\nabla\bfu,\nabla\bfvarphi\big\rangle\,\dif s+\int_0^t\big\langle\varPsi(\bfu)\,\dif W,\bfvarphi\big\rangle.
\end{align*}
\end{enumerate}
\end{definition}

\begin{theorem}\label{thm:inc2d}
Let $N=2$. Let $(\Omega,\mf,(\mf_t),\prst)$ be a given stochastic basis with an $(\mf_t)$-cylindrical Wiener process $W$ and let $\bfu_0$ be an $\mf_0$-measurable random variable such that $\bfu_0\in L^p(\Omega;L^2(\mathbb T^2))$ for all $p\in[1,\infty)$. Then there exists a unique weak pathwise solution to \eqref{eq:lim} with the initial condition $\bfu_0$.
\end{theorem}

The main results of the present paper are following.

\begin{theorem}\label{thm:main}
Let $\Lambda$ be a given Borel probability measure on $L^2(\mt)$. Let $\Lambda_\varepsilon$ be a Borel probability measure on $L^\gamma(\mt)\times L^\frac{2\gamma}{\gamma+1}(\mt)$ such that for some constant $M>0$ (independent of $\varepsilon$) it holds true that
\begin{equation*}
\Lambda_\varepsilon\bigg\{(\rho,\bfq)\in L^\gamma(\mt)\times L^\frac{2\gamma}{\gamma+1}(\mt);\, \rho\geq\frac{1}{M},\,(\rho)_{\mt}\leq M,\;\Big|\frac{\rho-1}{\varepsilon}\Big|\leq M\bigg\}=1,
\end{equation*}
for all $p\in[1,\infty)$,
\begin{equation*}%\label{initial}
\int_{L^\gamma_x\times L^\frac{2\gamma}{\gamma+1}_x}\bigg\|\frac{1}{2}\frac{|\bfq|^2}{\rho} \bigg\|_{L^1_x}^p\,\dif \Lambda_\varepsilon (\rho,\bfq)\leq C(p),
\end{equation*}
and that the marginal law of $\Lambda_\varepsilon$ corresponding to the second component converges to $\Lambda$ weakly in the sense of measures on $L^\frac{2\gamma}{\gamma+1}(\mt).$
If
$\big((\Omega^\varepsilon,\mf^\varepsilon,(\mf^\varepsilon),\prst^\varepsilon),\varrho_\varepsilon,\bfu_\varepsilon,W_\varepsilon\big)$ is a finite energy weak martingale solution to \eqref{eq:} with the initial law $\Lambda_\varepsilon$, $\varepsilon\in (0,1)$, then\footnote{If a topological space $X$ is equipped with the weak topology we write $(X,w)$.}
\begin{align*}
\varrho_\varepsilon&\rightarrow1\quad\text{in law on}\quad L^\infty(0,T;L^\gamma(\mt)),\\
\bfu_\varepsilon&\rightarrow\bfu\quad\text{in law on}\quad \big(L^2(0,T;W^{1,2}(\mt)),w\big),
\end{align*}
where $\bfu$ is a weak martingale solution to \eqref{eq:lim} with the initial law $\Lambda$.
\end{theorem}

\begin{theorem}\label{thm:main2d}
Let $N=2$ and $\bfu_0\in L^2(\mathbb T^2)$. Let $\Lambda_\varepsilon$ be a Borel probability measure on $L^\gamma(\mathbb T^2)\times L^\frac{2\gamma}{\gamma+1}(\mathbb T^2)$ such that for some constant $M>0$ (independent of $\varepsilon$) it holds true that
\begin{equation*}
\Lambda_\varepsilon\bigg\{(\rho,\bfq)\in L^\gamma(\mathbb T^2)\times L^\frac{2\gamma}{\gamma+1}(\mathbb T^2);\, \rho\geq\frac{1}{M},\;\Big|\frac{\rho-1}{\varepsilon}\Big|\leq M,\,(\rho)_{\mt}\leq M,\;\Big|\frac{\bfq-\bfu_0}{\varepsilon}\Big|\leq M\bigg\}=1,
\end{equation*}
%for all $p\in[1,\infty)$,
%\begin{equation*}%\label{initial}
%\int_{L^\gamma_x\times L^\frac{2\gamma}{\gamma+1}_x}\bigg\|\frac{1}{2}\frac{|\bfq|^2}{\rho} \bigg\|_{L^1_x}^p\,\dif \Lambda_\varepsilon (\rho,\bfq)\leq C(p),
%\end{equation*}
%and that the marginal law of $\Lambda_\varepsilon$ corresponding to the second component converges to $\Lambda$ weakly in the sense of measures on $L^\frac{2\gamma}{\gamma+1}(\mt).$
If $\big((\Omega,\mf,(\mf),\prst),\varrho_\varepsilon,\bfu_\varepsilon,W\big)$ is a finite energy weak martingale solution to \eqref{eq:} with the initial law $\Lambda_\varepsilon$, $\varepsilon\in (0,1)$, then
\begin{align*}
\varrho_\varepsilon&\rightarrow1\quad\text{in}\quad L^\infty(0,T;L^\gamma(\mathbb T^2))\quad \prst\text{-a.s.},\\
\bfu_\varepsilon&\rightarrow\bfu\quad\text{in}\quad\big(L^2(0,T;W^{1,2}(\mathbb T^2)),w\big)\quad \prst\text{-a.s.},
\end{align*}
where $\bfu$ is a weak pathwise solution to \eqref{eq:lim} with the initial condition $\bfu_0$.
\end{theorem}

Here and in the sequel, the letter $C$ denotes a constant that might change from one line to another and that is independent of $\varepsilon$.

%
%\begin{remark}
%Since it is well known (and corresponds to the deterministic setting), that the pathwise uniqueness holds true in space dimension $N=2$ for \eqref{eq:lim}, it is natural to ask whether in that case the convergence of the velocities in Theorem \ref{thm:main} can be strengthen, namely, if one can show that
%$$\bfu_\varepsilon\rightharpoonup\bfu\quad\text{in }\quad L^2(0,T;W^{1,2}(\mt))\quad\text{a.s.}$$
%using a Yamada-Watanabe-type method. However, our compactness argument is rather complex and involves, among others, also certain quantities defined in terms of the underlying acoustic wave equation. Consequently, such a Yamada-Watanabe argument would require, roughly speaking, to pass to the limit (as well as uniqueness for the corresponding limit problem) also in the gradient part of the original system \eqref{eq:}, which is not possible.
%\end{remark}

\section{Proof of Theorem \ref{thm:main}}
\label{sec:main}

This section is devoted to the study the limit $\varepsilon\rightarrow0$ in the system \eqref{eq:}.
%\begin{subequations}\label{eq:approx2}
% \begin{align}
%  \dif \varrho+\divergence(\varrho\bu)\dif t&=0,\label{eq:approx21}\\
%  \dif(\varrho\bu)+
%\big[\divergence(\varrho\bu\otimes\bu)-\nu\Delta\bu-(\lambda+\nu)\nabla\divergence\bfu
%+\frac{1}{\varepsilon^2}\nabla \varrho^\gamma\big]\dif t&=\varPhi(\varrho,\varrho\bu) \,\dif W,\label{eq:approx22}
%%  \varrho(0)=\varrho_0,\qquad (\varrho\bu (0)&=(\varrho\bu)_0\label{eq:approx3}
% \end{align}
%%\end{equation}
%\end{subequations}
%with the initial law $\Lambda_\varepsilon$, where \color{red}
%\begin{align*}
%\Lambda_\varepsilon\set{\varrho(0)=1+\varepsilon\varrho_0,\ \varrho_0 \in L^\infty(\Omega\times\mt), \bfu(0) \in L^{2}
%(\Omega \times \mt) } = 1.
%\end{align*}
%\color{black}
To this end, we recall that it was proved in \cite{BrHo} that for every $\varepsilon\in (0,1)$ there exists
$$\big((\Omega^\varepsilon,\mf^\varepsilon,(\mf^\varepsilon_t),\prst^\varepsilon),\varrho_\varepsilon,\bfu_\varepsilon, W_\varepsilon\big)$$
which is a weak martingale solution in the sense of Definition \ref{def:sol}. It was shown in \cite{jakubow} that it is enough to consider only one probability space, namely,
$$(\Omega^\varepsilon,\mf^\varepsilon,\prst^\varepsilon)=\big([0,1],\mathcal{B}([0,1]),\mathcal{L}\big)\qquad\forall \varepsilon\in (0,1)$$
where $\mathcal{L}$ denotes the Lebesgue measure on $[0,1]$.
Moreover, we can assume without loss of generality that there exists one common Wiener process $W$ for all $\varepsilon$. Indeed, one could perform the existence proof from \cite{BrHo} for all the parameters from any chosen subsequence $\varepsilon_n$ at once. The reader is referred to \cite{BrHo} where a similar issue is discussed at the beginning of Section 5.

\subsection{Uniform bounds}
\label{UB}

We start with an a priori estimate which is a modification of the energy estimate \eqref{energy} established in \cite{BrHo}.

\begin{proposition}\label{prop:apriori}
Let $p\in[1,\infty).$ Then the following estimate holds true uniformly in $\varepsilon$
\begin{align}\label{apriori}
\begin{aligned}
&\stred\bigg[\sup_{0\leq t\leq T}\int_ {\mt} \Big(\frac{1}{2} \varrho_\varepsilon(t)\big| \bfu_\varepsilon(t)\big|^2+\frac{1}{\varepsilon^2(\gamma-1)}\big(\varrho^\gamma_\varepsilon(t)-1-\gamma(\varrho_\varepsilon(t)-1)\big)\Big)\dx\bigg]^p\\
&\quad+\stred\bigg[\int_0^{T}\int_ {\mt}\nu |\nabla \bfu_\varepsilon|^2+(\lambda+\nu)|\divergence\bfu_\varepsilon|^2\bigg]^p\\
& \leq C_p\,\stred\bigg[\int_{\mt} \Big(\frac{1}{2} \varrho_{\varepsilon}(0)| \bfu_{\varepsilon}(0)|^2+\frac{1}{\varepsilon^2(\gamma-1)}\big(\varrho_{\varepsilon}^\gamma(0)-1-\gamma(\varrho_{\varepsilon}(0)-1)\big)\Big)\dif x+1\bigg]^p\\
&\leq C_p.
\end{aligned}
\end{align}
\end{proposition}

\begin{proof}
The first inequality follows directly from \eqref{energy} choosing
\[
H(\varrho) =  \frac{1}{\gamma - 1} \left( \varrho^{\gamma} - \gamma \overline{\varrho}^{\gamma - 1}(\varrho - 1) - 1\right)
\]
 and using $\int_{\mt}\varrho(0)\leq M$.
%and the mass conservation
%\begin{align*}
%\int_{\mt}\varrho_\varepsilon(t)\dx=\int_{\mt}\varrho_\varepsilon(0)\dx
%\end{align*}
%which is a consequence of equation \eqref{eq1}.
Next, we observe that due to the Taylor theorem and our assumptions upon $\Lambda_\varepsilon$, it holds
\begin{align*}
\E\int_{\mt}\big(\varrho_{\varepsilon}^\gamma(0)-1-\gamma(\varrho_{\varepsilon}(0)-1)\big)\dif x&\leq C\varepsilon^2
\end{align*}
and hence the second estimate follows (independently of $\varepsilon$).\end{proof}
\iffalse
Let us now take supremum in time, $p$-th power and expectation. For the stochastic integral $J_6$ we make use of the Burkholder-Davis-Gundy inequality and the assumption \eqref{growth1} to obtain, for all $t\in[0,T]$, for all $\kappa>0$
\begin{equation*}
\begin{split}
\E\sup_{0\leq s\leq t}&|J_8|^{p}\leq C\,\E\bigg[\int_0^{t}\sum_{k\geq1}\bigg(\int_ {\mt} \bfu_\varepsilon\cdot \bfg_k\big(\varrho_\varepsilon,\varrho_\varepsilon\bfu_\varepsilon\big)\dx\bigg)^2\dif s\bigg]^{\frac{p}{2}}\\
&\hspace{-.5cm}\leq C\,\stred\bigg[\int_0^{t}\int_{\mt}\varrho_\varepsilon|\bu_\varepsilon|^2\,\dif x\int_{\mt}\sum_{k\geq 1}\varrho_\varepsilon^{-1}|\bfg_k(\varrho_\varepsilon,\varrho_\varepsilon\bu_\varepsilon)|^2\,\dif x\,\dif s\bigg]^\frac{p}{2}\\
&\hspace{-.5cm}\leq \kappa\,\stred\bigg[\sup_{0\leq s\leq t}\int_{\mt}\varrho_\varepsilon|\bu_\varepsilon|^2\,\dif x\bigg]^p+C(\kappa)\,\stred\int_0^{t}\bigg(\int_{\mt}\big(\varrho_\varepsilon+\varrho_\varepsilon|\bu_\varepsilon|^2+\varrho_\varepsilon^\gamma\big)\,\dif x\bigg)^p\,\dif s.
\end{split}
\end{equation*}
All the other terms can be handled as before. After choosing $\kappa$ small enough and applying Gronwall's Lemma
we end up with the a priori estimate \eqref{apriori} which is independent of $\varepsilon$.\fi

\begin{corollary}\label{prop:new}
We have the following uniform bounds, for all $p\in[1,\infty)$,
\begin{align}
\nabla \bfu_\varepsilon&\in L^{p}(\Omega;L^2(0,T; L^2 ( \mt))),\label{apv-}\\
\sqrt{ \varrho_\varepsilon} \bfu_\varepsilon&\in L^{p}(\Omega;L^\infty(0,T;L^2( \mt))),\label{aprhov}\\
\bfu_\varepsilon &\in L^{p}(\Omega;L^2(0,T; W^{1,2} ( \mt))).\label{apv}
\end{align}
\end{corollary}
\begin{proof}
The estimates \eqref{apv-} and \eqref{aprhov} follow immediately from Proposition
\ref{prop:apriori}.
For \eqref{apv} we use in addition
\begin{align*}
\int_{\mt}\varrho_\varepsilon(t)\dx=\int_{\mt}\varrho_\varepsilon(0)\dx \geq \frac{1}{M} |\mt|
\end{align*}
which is a consequence of \eqref{eq1}
\end{proof}

Let us now introduce the essential and residual component of any function $h$:
\begin{align*}
h &=h_{\text{ess}}+h_{\text{res}},\\
h_{\text{ess}}&=\chi(\varrho_\varepsilon )h,\; \chi\in C_c^\infty(0, \infty),\; 0 \leq  \chi\leq  1,\; \chi = 1 \text{ on an open interval containing }1 ,\\
h_{\text{res}}&=(1-\chi(\varrho_\varepsilon))h.
\end{align*}
The following lemma will be useful.

\begin{lemma}\label{lem1}
Let $P(\rho):=\rho^\gamma-1-\gamma(\rho-1),$ with $\rho\in[0,\infty).$ Then there exist constants $C_1,\,C_2,\,C_3,\,C_4>0$ such that
\begin{itemize}
\item[$(i)$] $C_1|\rho-1|^2\leq P(\rho)\leq C_2|\rho-1|^2$ if $\rho\in\supp\chi$,
\item[$(ii)$] $P(\rho)\geq C_4$ if $\rho\notin \supp\chi$,
\item[$(iii)$] $P(\rho)\geq C_3\rho^\gamma$ if $\rho\notin \supp\chi$.

\end{itemize}

\end{lemma}
\begin{proof}
The first statement follows immediately from the Taylor theorem. The second one is a consequence of the fact that $P$ is strictly convex and attains its minimum at $\rho=1$. If $\rho\notin\supp\chi$ and $\rho\in [0,1)$ then the third statement is a consequence of the second one. Finally, we observe that the function
$\frac{P(\rho)}{\rho^\gamma}$
is increasing for large $\rho\in[1,\infty)$ and its value at $\rho=1$ is zero. This implies the remaining part of $(iii)$ and the proof is complete.
\end{proof}

Accordingly, using Lemma \ref{lem1} and \eqref{apriori}, we obtain the following uniform bounds, for all $p\in[1,\infty)$
\begin{align*}
\Big[\frac{\varrho_\varepsilon-1}{\varepsilon}\Big]_{\text{ess}}&\in L^p(\Omega;L^\infty(0,T;L^2(\mt))),\\
\Big[\frac{\varrho_\varepsilon-1}{\varepsilon}\Big]_{\text{res}}&\in L^p(\Omega;L^\infty(0,T;L^\gamma(\mt))).
\end{align*}
Therefore, setting $\varphi_\varepsilon:=\frac{1}{\varepsilon}(\varrho_\varepsilon-1)$, we deduce that uniformly in $\varepsilon$
\begin{align}
\varphi_\varepsilon&\in L^{p}(\Omega;L^\infty(0,T;L^{\min(\gamma,2)}( \mt))).\label{apphi}
\end{align}

As the next step, we want to show that
\begin{align}\label{conv:rho}
\varrho_\varepsilon\rightarrow1\quad\text{in}\quad L^p(\Omega;L^\infty(0,T;L^\gamma(\mt))),
\end{align}
which in particular leads to
\begin{align}
\varrho_\varepsilon&\in L^{p}(\Omega;L^\infty(0,T;L^\gamma(\mt))).\label{aprho}
\end{align}
Then, combining \eqref{aprhov}, \eqref{apv} and \eqref{aprho}, respectively, we deduce the uniform bounds, for all $p\in[1,\infty)$,
\begin{align}
  \varrho_\varepsilon\bfu_\varepsilon&\in L^{p}(\Omega;L^\infty(0,T;L^\frac{2\gamma}{\gamma+1}( \mt))),\label{estrhou2}\\
  \varrho_\varepsilon\bfu_\varepsilon\otimes\bfu_\varepsilon&\in L^{p}(\Omega;L^2(0,T;L^\frac{6\gamma}{4\gamma+3}( \mt))).\label{estrhouu}
\end{align}
%For $\varepsilon$ small enough we can assume that $$(\varrho_\varepsilon(t))_{\mt}=(\varrho_{0,\varepsilon}(t))_{\mt}=1+\varepsilon (\varrho_0)_{\mt}\in[1,\tfrac{3}{2}].$$

Let us now verify \eqref{conv:rho}. Since for all $\delta>0$ there exists $C_\delta>0$ such that
\begin{align*}
\rho^\gamma-1-\gamma(\rho-1)\geq C_\delta |\rho-1|^\gamma
\end{align*}
if $|\rho-1|\geq\delta$ and $\rho\geq0$, we obtain by \eqref{apriori}
\begin{align*}
\stred &\bigg[\sup_{0\leq t\leq T}\int_{\mt} |\varrho_\varepsilon-1|^\gamma\dx\bigg]^p=\stred \bigg[\sup_{0\leq t\leq T}\int_{\mt} \ind_{\{|\varrho_\varepsilon-1|\geq \delta\}} |\varrho_\varepsilon-1|^\gamma\dx\bigg]^p\\
&\qquad+\stred \bigg[\sup_{0\leq t\leq T}\int_{\mt} \ind_{\{|\varrho_\varepsilon-1|< \delta\}} |\varrho_\varepsilon-1|^\gamma\dx\bigg]^p\\
&\leq\,C_\delta\,\stred \bigg[\sup_{0\leq t\leq T}\int_{\mt} \big(\varrho_\varepsilon^\gamma-1-\gamma(\varrho_\varepsilon-1)\big)\dx\bigg]^p+C\delta^{\gamma p}\leq C_\delta\varepsilon^{2p}+C\delta^{\gamma p}.
\end{align*}
Letting first $\varepsilon\rightarrow0$ and then $\delta\rightarrow0$ yields
the claim.\\
%\rmk{is this still needed?}\color{red} We remark that \eqref{apriori} yields for $R>1$ the uniform bound
%\begin{align}\label{eq:23}
%\varphi_\varepsilon\ind_{\{\varrho_\varepsilon\leq R\}}\in %L^2(\Omega;L^2(0,T;L^2(\mt)))
%\end{align}
%using the inequality for all $\rho\leq R$
%\begin{align*}
%&\rho^\gamma+\gamma-1-\gamma(\rho-1)\geq\nu|\rho-1|^2
%\end{align*}
%for some $\nu=\nu(R,\gamma)>0$.
%\color{black}

%As we have (after passing to a subsequence)
%\begin{align}\label{conv:u}
%\bfu_\varepsilon\rightharpoonup\bfu\quad\text{in}\quad L^p(\Omega;L^2(0,T;W^{1,2}(\mt))),
%\end{align}
%due to \eqref{apv}, hence \eqref{conv:rho} implies
%\begin{align}\label{conv:rhou}
%\varrho_\varepsilon\bfu_\varepsilon \rightharpoonup\bfu\quad\text{in}\quad L^p(\Omega;L^2(0,T;L^\frac{2\gamma}{\gamma+1}(\mt))).
%\end{align}
%If we pass to the limit in the continuity equation, we see that $\divergence\bfu=0$, so
%\begin{align}\label{conv:divu}
%\divergence\bfu_\varepsilon\rightharpoonup0\quad\text{in}\quad L^p(\Omega;L^2(0,T;L^2(\mt))).
%\end{align}

\subsection{Acoustic equation}
\label{AW}

In order to proceed we need the Helmholtz projection $\mathcal P_H$ which projects $L^2(\mt)$ onto divergence free vector fields
\begin{align*}
L^2_{\divergence}(\mt):=\overline{C^\infty_{\divergence}(\mt)}^{\|\cdot\|_2}.
\end{align*}
Moreover, we set $\mathcal Q=\mathrm{Id}-\mathcal P_H$. Recall that $\mathcal P_H$ can be easily defined
in terms of the Fourier coefficients $a_k$ (cf. (\ref{trigo}), in particular it can be shown that both $\mathcal P_H$ and $\mathcal Q$ are
continuous in all $W^{l,q}(\mt)$-spaces, $l\in\mr$, $q\in(1,\infty)$.
%, we gain
%\begin{align}\label{conv:Pu}
%\mathcal P\bfu_\varepsilon &\rightharpoonup\bfu\quad\text{in}\quad L^p(\Omega;L^2(0,T;W^{1,2}(\mt))),\\
%\label{conv:Prhou}
%\mathcal P(\varrho_\varepsilon\bfu_\varepsilon) &\rightharpoonup\bfu\quad\text{in}\quad L^p(\Omega;L^2(0,T;L^\frac{2\gamma}{\gamma+1}(\mt))).
%\end{align}

Let us now project \eqref{eq2} onto the space of gradient vector fields. Then \eqref{eq:} rewrites as (using $\mathcal Q\nabla f=\nabla f$)
\begin{subequations}\label{eq:approximation2}
 \begin{align}
  \varepsilon\,\dif \varphi_\varepsilon+\divergence\mathcal Q(\varrho_\varepsilon\bu_\varepsilon)\dif t&=0,\label{eq:approximation21}\\
  \varepsilon\,\dif\mathcal Q(\varrho_\varepsilon\bu_\varepsilon)+
\gamma \nabla\varphi_\varepsilon\dif t&=\varepsilon\bfF_\varepsilon\,\dif t+\varepsilon\mathcal Q\varPhi(\varrho_\varepsilon,\varrho_\varepsilon\bu_\varepsilon) \,\dif W,\label{eq:approximation22}\\
\bfF_\varepsilon= \nu \Delta \mathcal Q \bfu_\varepsilon + (\lambda+\nu)\nabla\divergence\bfu_\varepsilon-\mathcal Q[\divergence(\varrho_\varepsilon\bfu_\varepsilon\otimes\bfu_\varepsilon)]-&\frac{1}{\varepsilon^2}\nabla[\varrho_\varepsilon^\gamma
-1-\gamma(\varrho_\varepsilon-1)]\nonumber.
%  \varrho(0)=\varrho_0,\qquad (\varrho\bu (0)&=(\varrho\bu)_0\label{eq:approx3}
 \end{align}
%\end{equation}
\end{subequations}
The system (\ref{eq:approximation2}) may be viewed as a stochastic version of Lighthill's acoustic analogy \cite{Light1,Light2} associated to the compressible
Navier-Stokes system. Note that Proposition \ref{prop:apriori} and \eqref{estrhouu} yield for $l>\frac{N}{2}+1$ using Sobolev's embedding
\begin{align}\label{eq:F}
\bfF_\varepsilon\in L^p(0,T;L^1(0,T;W^{-l,2}(\mt)))
\end{align}
uniformly in $\varepsilon$.
\iffalse
We introduce the group $(\mathcal S(\tau))_{\tau\in\R}$ generated by the operator $L$, where
\begin{align*}
L\begin{pmatrix} \varphi \\ \bfv \end{pmatrix}=-\begin{pmatrix} \divergence\bfv\\ \gamma\nabla\varphi \end{pmatrix}.
\end{align*}
As shown in \cite[Section III]{LiMa}, the operator $\mathcal S(\tau)$, $\tau\in\mr$, is an isometry
on $W^{s,2}(\mt)\times W^{s,2}(\mt)^N$ endowed with the norm
\begin{align*}
\|(\varphi,\bfv)\|=\Big(\|\varphi\|^2_{W^{s,2}(\mt)}+\gamma^{-1}\|\bfv\|^2_{W^{s,2}(\mt)}\Big)^{\frac{1}{2}}
\end{align*}
for all $s\in\R$.
We set
\begin{align*}
\psi_\varepsilon=\mathcal S_1\Big(-\frac{t}{\varepsilon}\Big)\begin{pmatrix} \varphi_\varepsilon \\ \mathcal Q(\varrho_\varepsilon\bfu_\varepsilon) \end{pmatrix},\quad \bfm_\varepsilon=\mathcal S_2\Big(-\frac{t}{\varepsilon}\Big)\begin{pmatrix} \varphi_\varepsilon \\ \mathcal Q(\varrho_\varepsilon\bfu_\varepsilon) \end{pmatrix},
\end{align*}
where we split $\mathcal S$ into its components.
\fi

\subsection{Compactness}
\label{subsec:compactness}

Let us define the path space $\mathcal{X}=\mathcal{X}_\varrho\times\mathcal{X}_\bu\times\mathcal{X}_{\varrho\bu}\times\mathcal{X}_W$ where
\begin{align*}
\mathcal{X}_\varrho&=C_w(0,T;L^\gamma(\mt)),&\mathcal{X}_\bu&=\big(L^2(0,T;W^{1,2}(\mt)),w\big),\\
\mathcal{X}_{\varrho\bu}&=C_w([0,T];L^\frac{2\gamma}{\gamma+1}(\mt)),%&\mathcal{X}_{(\psi,\bfm)}&=L^2(0,T;W^{-1,2}(\mt)),\\
&\mathcal{X}_W&=C([0,T];\mathfrak{U}_0).
\end{align*}
Let us denote by $\mu_{\varrho_\varepsilon}$, $\mu_{\bu_\varepsilon}$ and $\mu_{\mathcal P(\varrho_\varepsilon\bu_\varepsilon)}$, %$\mu_{(\psi_\varepsilon,\bfm_\varepsilon)}$,
 respectively, the law of $\varrho_\varepsilon$, $\bu_\varepsilon$, $\mathcal P(\varrho_\varepsilon\bu_\varepsilon)$ %and $(\psi_\varepsilon,\bfm_\varepsilon)$
 on the corresponding path space. By $\mu_W$ we denote the law of $W$ on $\mathcal{X}_W$ and their joint law on $\mathcal{X}$ is denoted by $\mu^\varepsilon$.

To proceed, it is necessary to establish tightness of $\{\mu^\varepsilon;\,\varepsilon\in(0,1)\}$.

\begin{proposition}\label{prop:bfutightness}
The set $\{\mu_{\bu_\varepsilon};\,\varepsilon\in(0,1)\}$ is tight on $\mathcal{X}_\bu$.
\end{proposition}
\begin{proof}
This is a consequence of \eqref{apv}. Indeed, for any $R>0$ the set
$$B_R=\big\{\bu\in L^2(0,T;W^{1,2}(\mt));\, \|\bu\|_{L^2(0,T;W^{1,2}(\mt))}\leq R\big\}$$
is relatively compact in $\mathcal{X}_\bu$ and
\begin{equation*}
 \begin{split}
  \mu_{\bu_\varepsilon}(B_R^c)=\prst\big(\|\bu_\varepsilon\|_{L^2(0,T;W^{1,2}(\mt))}\geq R\big)\leq\frac{1}{R}\stred\|\bu_\varepsilon\|_{L^2(0,T;W^{1,2}(\mt))}\leq \frac{C}{R}
 \end{split}
\end{equation*}
which yields the claim.
\end{proof}

\begin{proposition}\label{prop:rhotight}
The set $\{\mu_{\varrho_\varepsilon};\,\varepsilon\in(0,1)\}$ is tight on $\mathcal{X}_\varrho$.
\end{proposition}
\begin{proof}
Due to \eqref{estrhou2}, $\{\divergence(\varrho_\varepsilon\bu_\varepsilon)\}$ is bounded in $L^{p}(\Omega;L^\infty(0,T;W^{-1,\frac{2\gamma}{\gamma+1}}(\mt)))$ and therefore the continuity equation yields the following uniform bound, for all $p\in[1,\infty),$
$$\varrho_\varepsilon\in L^p(\Omega;C^{0,1}([0,T];W^{-1,\frac{2\gamma}{\gamma+1}}(\mt)).$$
Now, the required tightness in follows by a similar reasoning as in Proposition \ref{prop:bfutightness} together with \eqref{aprho} and the compact embedding (see \cite[Corollary B.2]{on2})
$$L^\infty(0,T;L^\gamma(\mt))\cap C^{0,1}([0,T];W^{-2,\frac{2\gamma}{\gamma+1}}(\mt))\overset{c}{\hookrightarrow} C_w([0,T];L^\gamma(\mt)).$$
\end{proof}

\begin{proposition}\label{rhoutight1}
The set $\{\mu_{\mathcal P_H(\varrho_\varepsilon\bu_\varepsilon)};\,\varepsilon\in(0,1)\}$ is tight on $\mathcal{X}_{\varrho\bu}$.
\end{proposition}
\begin{proof}
We decompose $\mathcal P_H(\varrho_\varepsilon\bu_\varepsilon)$ into two parts, namely, $\mathcal P_H(\varrho_\varepsilon\bu_\varepsilon)(t)=Y^\varepsilon(t)+Z^\varepsilon(t)$, where
\begin{equation*}
 \begin{split}
Y^\varepsilon(t)&=\mathcal P_H\bfq_\varepsilon(0)-\int_0^t\mathcal P_H\big[\divergence(\varrho_\varepsilon\bu_\varepsilon\otimes\bu_\varepsilon)
-\nu\Delta\bu_\varepsilon
\big]\dif s,\\
Z^\varepsilon(t)&=\int_0^t\,\mathcal P_H \varPhi(\varrho_\varepsilon,\varrho_\varepsilon\bu_\varepsilon) \,\dif W(s).
 \end{split}
\end{equation*}
{\em H\"older continuity of $(Y^\varepsilon)$.}
We show that there exists $l\in\mn$ such that for all $\kappa\in(0,1/2)$ it holds true
\begin{equation}\label{eq:holderZ2}
\stred\|Y^\varepsilon\|_{C^\kappa([0,T];W^{-l,2}(\mt))}\leq C.
\end{equation}
Choose $l$ such that $L^1(\mt)\hookrightarrow W^{-l+1,2}(\mt)$. The a priori estimates \eqref{apv} and \eqref{estrhouu} and the continuity of $\mathcal P$ yield
\begin{equation*}
\begin{split}
\stred&\,\bigg\|Y^\varepsilon(t)-Y^\varepsilon(s)\bigg\|^\theta_{W^{-l,2}(\mt)}= \,\stred\,\bigg\|\int_s^t\mathcal P\big[\divergence(\varrho_\varepsilon\bu_\varepsilon\otimes\bu_\varepsilon)
+\nu\Delta\bu_\varepsilon
\big]\dif s\bigg\|^\theta_{W^{-l,2}(\mt)}\\
&\leq\,C\,\stred\,\bigg\|\int_s^t\divergence(\varrho_\varepsilon\bu_\varepsilon\otimes\bu_\varepsilon)
\,\dif s\bigg\|^\theta_{W^{-l,2}(\mt)}+\,C\,\stred\,\bigg\|\int_s^t\Delta\bfu_\varepsilon\,\dif s\bigg\|^\theta_{W^{-l,2}(\mt)}\\
&\leq\,C\,\stred\,\bigg\|\int_s^t\varrho_\varepsilon\bu_\varepsilon\otimes\bu_\varepsilon
\,\dif s\bigg\|^\theta_{L^{1}(\mt)}+\,C\,\stred\,\bigg\|\int_s^t\nabla\bfu_\varepsilon\,\dif s\bigg\|^\theta_{L^{1}(\mt)}\leq C|t-s|^{\theta/2}
\end{split}
\end{equation*}
and \eqref{eq:holderZ2} follows by the Kolmogorov continuity criterion.

{\em H\"older continuity of $(Z^\varepsilon)$.}
Next, we show that also
\begin{equation*}
\stred\|Z^\varepsilon\|_{C^\kappa([0,T];W^{-l,2}(\mt))}\leq C,
\end{equation*}
where $l\in\mn$ was given by the previous step and $\kappa\in(0,1/2)$.
From the embedding $L^1(\mt)\hookrightarrow W^{-l,2}(\mt)$, \eqref{growth1-}, \eqref{growth1}, the a priori estimates and the continuity of $\mathcal P_H$ we get
\begin{equation*}
\begin{split}
\stred&\,\bigg\|Z^\varepsilon(t)-Z^\varepsilon(s)\bigg\|^\theta_{W^{-l,2}(\mt)}= \,\stred\,\bigg\|\int_s^t\mathcal P_H\varPhi(\varrho_\varepsilon,\varrho_\varepsilon\bu_\varepsilon)\,\dif W\bigg\|^\theta_{W^{-l,2}(\mt)}\\
&\leq\,C\,\stred\,\bigg\|\int_s^t\varPhi(\varrho_\varepsilon,\varrho_\varepsilon\bu_\varepsilon)\,\dif W\bigg\|^\theta_{W^{-l,2}(\mt)}\leq C\,\stred\bigg(\int_s^t\sum_{k\geq1}\big\| \bfg_k(\varrho_\varepsilon,\varrho_\varepsilon\bu_\varepsilon)\big\|_{W^
{-l,2}}^2\,\dif r\bigg)^\frac{\theta}{2}\\
&\leq C\,\stred\bigg(\int_s^t\sum_{k\geq1}\big\|\bfg_k(\varrho_\varepsilon,\varrho_\varepsilon\bu_\varepsilon)\big\|_{L^
{1}}^2\,\dif r\bigg)^\frac{\theta}{2}\\
&\leq C\,\stred\bigg(\int_s^t\int_{\mt}(\varrho_\varepsilon+\varrho_\varepsilon|\bu_\varepsilon|^2+\varrho_\varepsilon^\gamma)\,\dif x\,\dif r\bigg)^{\frac{\theta}{2}}\\
&\leq C|t-s|^{\frac{\theta}{2}}\Big(1+\stred\sup_{0\leq t\leq T}\|\sqrt\varrho_\varepsilon\bu_\varepsilon\|_{L^{2}}^{\theta}+\stred\sup_{0\leq t\leq T}\|\varrho_\varepsilon\|_{L^\gamma}^{\theta\gamma/2}\Big)\leq C|t-s|^{\frac{\theta}{2}}
\end{split}
\end{equation*}
and the Kolmogorov continuity criterion applies.

{\em Conclusion.}
Collecting the above results we obtain that
$$\stred\|\mathcal P_H(\varrho_\varepsilon\bfu_\varepsilon)\|_{C^\kappa([0,T];W^{-l,2}(\mt)}\leq C$$
for some $l\in\mn$ and all $\kappa\in(0,1/2)$. This implies the desired tightness by making use of \eqref{estrhou2}, continuity of $\mathcal{P}_H$ together with the compact embedding (see \cite[Corollary B.2]{on2})
$$L^\infty(0,T;L^\frac{2\gamma}{\gamma+1}(\mt))\cap C^{\kappa}([0,T];W^{-l,2}(\mt))\overset{c}{\hookrightarrow} C_w([0,T];L^\frac{2\gamma}{\gamma+1}(\mt)).$$
\end{proof}

Since also the law$\mu_W$ is tight as being Radon measures on the Polish space $\mathcal{X}_W$ we can finally deduce tightness of the joint laws $\mu^\varepsilon$.

\begin{corollary}\label{cor:tight}
The set $\{\mu^\varepsilon;\,\varepsilon\in(0,1)\}$ is tight on $\mathcal{X}$.
\end{corollary}

The path space $\mathcal{X}$ is not a Polish space and so our compactness argument is based on the Jakubowski-Skorokhod representation theorem instead of the classical Skorokhod representation theorem, see \cite{jakubow}. To be more precise, passing to a weakly convergent subsequence $\mu^\varepsilon$ (and denoting by $\mu$ the limit law) we infer the following result.

\begin{proposition}\label{prop:skorokhod1}
There exists a subsequence $\mu^\varepsilon$, a probability space $(\tilde\Omega,\tilde\mf,\tilde\prst)$ with $\mathcal{X}$-valued Borel measurable random variables $(\tilde\varrho_\varepsilon,\tilde\bu_\varepsilon,\tilde\bq_\varepsilon,\tilde W_\varepsilon)$, $n\in\mn$, and $(\tilde\varrho,\tilde\bu,\tilde\bq,\tilde W)$ such that
\begin{enumerate}
 \item the law of $(\tilde\varrho_\varepsilon,\tilde\bu_\varepsilon,\tilde\bq_\varepsilon,\tilde W_\varepsilon)$ is given by $\mu^\varepsilon$, $\varepsilon\in(0,1)$,
\item the law of $(\tilde\varrho,\tilde\bu,\tilde\bq,\tilde W)$, denoted by $\mu$, is a Radon measure,
 \item $(\tilde\varrho_\varepsilon,\tilde\bu_\varepsilon,\tilde\bq_\varepsilon,\tilde W_\varepsilon)$ converges $\,\tilde{\prst}$-a.s. to $(\tilde\varrho,\tilde{\bu},\tilde\bq,\tilde{W})$ in the topology of $\mathcal{X}$.
\end{enumerate}
\end{proposition}

\color{red}

%Remark EF:
%In the case of unbounded domain (the whole space $R^N$) we should be able to show directly in Proposition 3.5 that
%$\mu_{(\varrho_\varepsilon \bfu_{\varepsilon})}$ is tight in $X_{\varrho \bfu}$...

\color{black}

Let us now fix some notation that will be used in the sequel. We denote by $\bfr_t$ the operator of restriction to the interval $[0,t]$ acting on various path spaces. In particular, if $X$ stands for one of the path spaces $\mathcal{X}_\varrho,\,\mathcal{X}_{\bfu}$ or $\mathcal{X}_{W}$ and $t\in[0,T]$, we define%\marginpar{\textcolor{red}{what about $\mathcal X_\bfq$?}\rmk{this doesn't have to be included in filtration as $\bfq=\varrho\bfu$ is a measurable function of $\varrho$ and $\bfu$, who are already in the filtration}}
\begin{align}\label{restr}
\bfr_t:X\rightarrow X|_{[0,t]},\quad f\mapsto f|_{[0,t]}.
\end{align}
Clearly, $ \bfr_t$ is a continuous mapping.
Let $(\tilde{\mf}_t^\varepsilon)$ and $(\tilde{\mf}_t)$, respectively, be the $\tilde{\prst}$-augmented canonical filtration of the process $(\tilde\varrho_\varepsilon,\tilde{\bu}_\varepsilon,\tilde{W}_\varepsilon)$ and $(\tilde\varrho,\tilde{\bu},\tilde{W})$, respectively, that is
\begin{equation*}
\begin{split}
\tilde{\mf}_t^\varepsilon&=\sigma\big(\sigma\big(\bfr_t\tilde\varrho_\varepsilon,\,\bfr_t\tilde{\bu}_\varepsilon,\,\bfr_t \tilde{W}_\varepsilon\big)\cup\big\{N\in\tilde{\mf};\;\tilde{\prst}(N)=0\big\}\big),\quad t\in[0,T],\\
\tilde{\mf}_t&=\sigma\big(\sigma\big(\,\bfr_t\tilde{\bu},\,\bfr_t\tilde{W}\big)\cup\big\{N\in\tilde{\mf};\;\tilde{\prst}(N)=0\big\}\big),\quad t\in[0,T].
\end{split}
\end{equation*}

\subsection{Identification of the limit}
\label{subsec:ident}

The aim of this subsection is to identify the limit processes given by Proposition \ref{prop:skorokhod1} with a weak martingale solution to \eqref{eq:lim}. Namely, we prove the following result which in turn verifies Theorem \ref{thm:main}.

\begin{theorem}\label{thm:1}
The process $\tilde W$ is a $(\tilde\mf_t)$-cylindrical Wiener process and
$$\big((\tilde\Omega,\tilde\mf,(\tilde\mf_t),\tilde\prst),\tilde \bfu,\tilde W\big)$$
is a weak martingale solution to \eqref{eq:lim} with the initial law $\Lambda$.

\end{theorem}

The proof proceeds in several steps. First of all, we show that also on the new probability space $(\tilde\Omega,\tilde\mf,\tilde\p)$, the approximations $\tilde\varrho_\varepsilon,\,\tilde\bfu_\varepsilon$ solve the corresponding compressible Navier-Stokes system \eqref{eq:}.

\begin{proposition}\label{prop:limit1}
Let $\varepsilon\in(0,1)$. The process $\tilde{W}_\varepsilon$ is a $(\tilde{\mf}_t)$-cylindrical Wiener process and
$$\big((\tilde{\Omega},\tilde{\mf},(\tilde{\mf}_t^\varepsilon),\tilde{\prst}),\tilde\varrho_\varepsilon,\tilde{\bu}_\varepsilon,\tilde{W}_\varepsilon\big)$$
is a finite energy weak martingale solution to \eqref{eq:} with initial law $\Lambda_\varepsilon$.
\end{proposition}

\begin{proof}

The first part of the claim follows immediately form the fact that $\tilde W_\varepsilon$ has the same law as $W$. As a consequence, there exists a collection of mutually independent real-valued $(\tilde{\mf}_t)$-Wiener processes $(\tilde{\beta}^{\varepsilon}_k)_{k\geq1}$ such that $\tilde{W}_\varepsilon=\sum_{k\geq1}\tilde{\beta}^{\varepsilon}_k e_k$.

To show that the continuity equation \eqref{eq1} is satisfied, let us define, for all $t\in[0,T]$ and $\psi\in C^\infty(\mt)$, the functional
$$L(\rho,\bfq)_t=\langle\rho(t),\psi\rangle-\langle\rho(0),\psi\rangle-\int_0^t\langle\bfq,\nabla\psi\rangle\,\dif s.$$
Note that $(\rho,\bfq)\mapsto L(\rho,\bfq)_t$ is continuous on $\mathcal{X}_\varrho\times\mathcal{X}_{\varrho\bu}$. Hence the laws of $L(\varrho_\varepsilon,\varrho_\varepsilon\bfu_\varepsilon)_t$ and $L(\tilde\varrho_\varepsilon,\tilde\varrho_\varepsilon\tilde\bfu_\varepsilon)_t$ coincide and since $(\varrho_\varepsilon,\varrho_\varepsilon\bfu_\varepsilon)$ solves \eqref{eq1} we deduce that
\begin{align*}
\tilde\stred\big|L(\tilde\varrho_\varepsilon,\tilde\varrho_\varepsilon\tilde\bfu_\varepsilon)_t\big|^2=\stred\big|L(\varrho_\varepsilon,\varrho_\varepsilon\bfu_\varepsilon)_t\big|^2=0
\end{align*}
hence $(\tilde\varrho_\varepsilon,\tilde\varrho_\varepsilon\tilde\bfu_\varepsilon)$ solves \eqref{eq1}.

To verify the momentum equation \eqref{eq2}, we define for all $t\in[0,T]$ and $\bfvarphi\in C^\infty(\mt)$ the functionals
\begin{equation*}
\begin{split}
M(\rho,\bfv,\bfq)_t&=\big\langle\bfq(t),\bfvarphi\big\rangle-\big\langle \bfq(0),\bfvarphi\big\rangle+\int_0^t\big\langle\bfq\otimes\bfv,\nabla\bfvarphi\big\rangle\,\dif r-\nu\int_0^t\big\langle\nabla\bfv,\nabla\bfvarphi\big\rangle\,\dif r\\
&\quad-(\lambda+\nu)\int_0^t\big\langle\divergence\bfv,\divergence\bfvarphi\big\rangle\,\dif r+\frac{a}{\varepsilon^2}\int_0^t\big\langle\rho^\gamma,\divergence\bfvarphi\big\rangle\,\dif r\\
N(\rho,\bfq)_t&=\sum_{k\geq1}\int_0^t\big\langle \bfg_k(\rho, \bfq ),\bfvarphi\big\rangle^2\,\dif r,\\
N_k(\rho,\bfq)_t&=\int_0^t\big\langle \bfg_k(\rho,\bfq),\bfvarphi\big\rangle\,\dif r,
\end{split}
\end{equation*}
let $M(\rho,\bfv,\bfq)_{s,t}$ denote the increment $M(\rho,\bfv,\bfq)_{t}-M(\rho,\bfv,\bfq)_{s}$ and similarly for $N(\rho,\bfq)_{s,t}$ and $N_k(\rho,\bfq)_{s,t}$.
We claim that with the above uniform estimates in hand, the mappings
$$(\rho,\bfv,\bfq)\mapsto M(\rho,\bfv,\bfq)_t,\quad\,(\rho,\bfv,\bfq)\mapsto N(\rho,\bfq)_t,\,\quad (\rho,\bfv,\bfq)\mapsto N_k(\rho,\bfq)_t$$
are well-defined and measurable on a subspace of $\mathcal{X}_\varrho\times\mathcal{X}_\bu\times\mathcal{X}_{\varrho\bfu}$ where the joint law of $(\tilde\varrho,\tilde\bfu,\tilde\varrho\tilde\bfu)$ is supported, i.e. where all the uniform estimates hold true.
Indeed, in the case of $N(\rho,\bfq)_t$ we have by \eqref{growth1-}, \eqref{growth1} similarly to \eqref{stochest}
\begin{align*}
\sum_{k\geq 1}\int_0^t\big\langle \bfg_k(\rho,\bfq),\varphi\big\rangle^2\,\dif s&\leq C\sum_{k\geq 1}\int_0^t\|\, \bfg_k(\rho,\bfq)\|_{L^1}^2\,\dif s\leq C.
%&\leq C\int_0^t\int_{\mt}\Big(\rho +\rho^{\gamma}+\frac{|\bfq|^2}{\rho}\Big)\,\dif x\,\dif s
\end{align*}
%which is finite due to \eqref{aprho1} and \eqref{est:rhobfu2}.
$M(\rho,\bfv,\bfq)$ and $N_k(\rho,\bfv)_t\,$ can be handled similarly and therefore, the following random variables have the same laws
\begin{align*}
M(\varrho_\varepsilon,\bu_\varepsilon,\varrho_\varepsilon\bfu_\varepsilon)&\distr M(\tilde\varrho_\varepsilon,\tilde\bu_\varepsilon,\tilde\varrho_\varepsilon\tilde\bfu_\varepsilon),\\
N(\varrho_\varepsilon,\varrho_\varepsilon\bfu_\varepsilon)&\distr N(\tilde\varrho_\varepsilon,\tilde\varrho_\varepsilon\tilde\bfu_\varepsilon),\\
N_k(\varrho_\varepsilon,\varrho_\varepsilon\bfu_\varepsilon)&\distr N_k(\tilde\varrho_\varepsilon,\tilde\varrho_\varepsilon\tilde\bfu_\varepsilon).
\end{align*}

Let us now fix times $s,t\in[0,T]$ such that $s<t$ and let
$$h:\mathcal{X}_\varrho|_{[0,s]}\times\mathcal{X}_{\bfu}|_{[0,s]}\times\mathcal{X}_W|_{[0,s]}\rightarrow [0,1]$$
be a continuous function.
Since
$$M(\varrho_\varepsilon,\bu_\varepsilon,\varrho_\varepsilon\bfu_\varepsilon)_t=\int_0^t\big\langle\varPhi(\varrho_\varepsilon,\varrho_\varepsilon\bu_\varepsilon)\,\dif W,\bfvarphi\big\rangle=\sum_{k\geq1}\int_0^t\big\langle \bfg_k(\varrho_\varepsilon,\varrho_\varepsilon\bu_\varepsilon),\bfvarphi\big\rangle\,\dif\beta_k$$
is a square integrable $(\mf_t)$-martingale, we infer that
$$\big[M(\varrho_\varepsilon,\bu_\varepsilon,\varrho_\varepsilon\bfu_\varepsilon)\big]^2-N(\varrho_\varepsilon,\varrho_\varepsilon\bfu_\varepsilon),\quad M(\varrho_\varepsilon,\bu_\varepsilon,\varrho_\varepsilon\bfu_\varepsilon)\beta_k-N_k(\varrho_\varepsilon,\varrho_\varepsilon\bfu_\varepsilon)$$
are $(\mf_t)$-martingales.
Besides, it follows from the equality of laws that
\begin{equation}\label{exp11}
\begin{split}
&\tilde{\stred}\,h\big(\bfr_s\tilde\varrho_\varepsilon, \bfr_s\tilde{\bu}_\varepsilon,\bfr_s\tilde{W}_\varepsilon\big)\big[M(\tilde\varrho_\varepsilon,\tilde\bu_\varepsilon,\tilde\varrho_\varepsilon\tilde\bfu_\varepsilon)_{s,t}\big]\\
&=\stred\,h\big(\bfr_s\varrho_\varepsilon, \bfr_s\bu_\varepsilon, \bfr_s W_\varepsilon\big)\big[M(\varrho_\varepsilon,\bu_\varepsilon,\varrho_\varepsilon\bfu_\varepsilon)_{s,t}\big]=0,
\end{split}
\end{equation}
\begin{equation}\label{exp21}
\begin{split}
&\tilde{\stred}\,h\big(\bfr_s\tilde\varrho_\varepsilon, \bfr_s\tilde{\bu}_\varepsilon,\bfr_s\tilde{W}_\varepsilon\big)\bigg[[M(\tilde\varrho_\varepsilon,\tilde\bu_\varepsilon,\tilde\varrho_\varepsilon\tilde\bfu_\varepsilon)^2]_{s,t}-N(\tilde\varrho_\varepsilon,\tilde\varrho_\varepsilon\tilde\bfu_\varepsilon)_{s,t}\bigg]\\
&=\stred\,h\big(\bfr_s\varrho_\varepsilon, \bfr_s\bu_\varepsilon, \bfr_s W_\varepsilon\big)\bigg[[M(\varrho_\varepsilon,\bu_\varepsilon,\varrho_\varepsilon\bfu_\varepsilon)^2]_{s,t}-N(\varrho_\varepsilon,\varrho_\varepsilon\bfu_\varepsilon)_{s,t}\bigg]=0,
\end{split}
\end{equation}
\begin{equation}\label{exp31}
\begin{split}
&\tilde{\stred}\,h\big(\bfr_s\tilde\varrho_\varepsilon, \bfr_s\tilde{\bu}_\varepsilon,\bfr_s\tilde{W}_\varepsilon\big)\bigg[[M(\tilde\varrho_\varepsilon,\tilde\bu_\varepsilon,\tilde\varrho_\varepsilon\tilde\bfu_\varepsilon)\tilde{\beta}_k^\varepsilon]_{s,t}-N_k(\tilde\varrho_\varepsilon,\tilde\varrho_\varepsilon\tilde\bfu_\varepsilon)_{s,t}\bigg]\\
&=\stred\,h\big(\bfr_s\varrho_\varepsilon, \bfr_s\bu_\varepsilon, \bfr_s W_\varepsilon\big)\bigg[[M(\varrho_\varepsilon,\bu_\varepsilon,\varrho_\varepsilon\bfu_\varepsilon)\beta_k]_{s,t}-N_k(\varrho_\varepsilon,\varrho_\varepsilon\bfu_\varepsilon)_{s,t}\bigg]=0.
\end{split}
\end{equation}
The proof is hereby complete.
\end{proof}

Consequently, we recover the result of Proposition \ref{prop:apriori} together with all the uniform estimates of the previous subsection. In particular, we find (for a subsequence) that
\begin{equation}\label{conv:rho2}
\tilde\varrho_\varepsilon\rightarrow 1\quad\text{in}\quad L^\infty(0,T;L^\gamma(\mt))\quad\tilde\prst\text{-a.s.}
\end{equation}
Due to Corollary \ref{prop:new} we have the following bounds on the new probability space.
\begin{corollary}
\label{cor:limit1}
We have the following bounds uniform in $\varepsilon$, for all $p\in [1,\infty)$ and $l>\frac{N}{2}$,
\begin{align*}
\sqrt{\tilde\varrho_\varepsilon}\tilde\bfu_\varepsilon&\in L^p(\Omega,L^\infty(0,T;L^{2}(\mt))),\\
\tilde\varphi_\varepsilon&\in L^p(\Omega,L^\infty(0,T;L^{\min(2,\gamma)}(\mt))),\\
%\tilde\bfu_\varepsilon&\in L^p(\Omega,L^2(0,T;W^{1,2}(\mt)))\\
%\tilde\varrho_\varepsilon\tilde\bfu_\varepsilon&\in L^p(\Omega,L^\infty(0,T;L^{\frac{2\gamma}{\gamma+1}}(\mt))),\\
\tilde\bfF_\varepsilon&\in L^p(0,T;L^1(0,T;W^{-l,2}(\mt)))
\end{align*}
where $\tilde\varphi_\varepsilon=\tfrac{\tilde\varrho_\varepsilon-1}{\varepsilon}$ and
\begin{align*}
\tilde\bfF_\varepsilon= \nu \Delta \mathcal Q \bfu_\varepsilon + (\lambda+\nu)\nabla\divergence\tilde\bfu_\varepsilon-\mathcal Q[\divergence(\tilde\varrho_\varepsilon\tilde\bfu_\varepsilon\otimes\tilde\bfu_\varepsilon)]-&\frac{1}{\varepsilon^2}\nabla[\tilde\varrho_\varepsilon^\gamma
-1-\gamma(\tilde\varrho_\varepsilon-1)].
\end{align*}
\end{corollary}

\begin{proposition}\label{prop:limitnew}
We have the following convergence $\tilde\p$-a.s.
\begin{align}\label{conv:newu}
\mathcal P_H\tilde\bfu_\varepsilon &\rightarrow\tilde\bfu\quad\text{in}\quad L^2(0,T;L^{q}(\mt))\quad\forall q<\tfrac{2N}{N-2}.
\end{align}
\end{proposition}
\begin{proof}
Since the joint laws of $(\varrho_\varepsilon,\bfu_\varepsilon,\mathcal P_H(\varrho_\varepsilon\bfu_\varepsilon))$ and $(\tilde\varrho_\varepsilon,\tilde\bfu_\varepsilon,\tilde\bfq_\varepsilon)$ coincide, we deduce that $\tilde\bfq_\varepsilon=\mathcal P_H(\tilde\varrho_\varepsilon\tilde\bfu_\varepsilon)$ a.s. and consequently it follows from the proof of Proposition \ref{rhoutight1} that
\begin{equation}\label{eq:holderZ}
\tilde\stred\|\mathcal P_H(\tilde\varrho_\varepsilon\tilde{\bfu}_\varepsilon)\|_{C^\kappa([0,T];W^{-l,2}(\mt))}\leq C
\end{equation}
for some $\kappa\in(0,1)$ and $l\in\N$.

Besides, it follows from \eqref{conv:rho2} and the convergence of $\tilde\bfu_\varepsilon$ to $\tilde\bfu$ that
\begin{align}\label{conv:rhou}
\tilde\varrho_\varepsilon\tilde\bfu_\varepsilon \rightharpoonup\tilde\bfu\quad\text{in}\quad L^2(0,T;L^\frac{2\gamma}{\gamma+1}(\mt))\quad\tilde\prst\text{-a.s}.
\end{align}
If we pass to the limit in the continuity equation, we see that $\divergence\tilde\bfu=0$,
which in turn identifies $\tilde\bfq$ with $\tilde\bfu$. Indeed, due to continuity of $\mathcal{P}$ we obtain
\begin{align*}
\mathcal P_H(\tilde\varrho_\varepsilon\tilde\bfu_\varepsilon) &\rightharpoonup\tilde\bfu\quad\text{in}\quad L^2(0,T;L^\frac{2\gamma}{\gamma+1}(\mt))\quad\tilde\prst\text{-a.s}.
\end{align*}
Thus with Proposition \ref{prop:skorokhod1} and the compact embedding $L^{\frac{2\gamma}{\gamma+1}}(\mt)\overset{c}{\hookrightarrow} W^{-1,2}(\mt)$
\begin{align}
\label{conv:Prhou}
\mathcal P_H(\tilde\varrho_\varepsilon\tilde\bfu_\varepsilon) &\rightarrow\tilde\bfu\quad\text{in}\quad L^2(0,T;W^{-1,2}(\mt))\quad\tilde\prst\text{-a.s}.
\end{align}
Since
\begin{align}\label{conv:divu}
\divergence\tilde\bfu_\varepsilon\rightharpoonup0\quad\text{in}\quad L^2(0,T;L^2(\mt))\quad\tilde\prst\text{-a.s}.
\end{align}
we have also that
\begin{align}\label{conv:Pu}
\mathcal P_H\tilde\bfu_\varepsilon &\rightharpoonup\tilde\bfu\quad\text{in}\quad L^2(0,T;W^{1,2}(\mt))\quad\tilde\prst\text{-a.s}.
\end{align}
Note that \eqref{conv:divu} is a consequence of $\diver\tilde\bfu=0$ and the $\tilde\prst$-a.s. convergence $\tilde\bfu_\varepsilon\rightharpoonup \tilde\bfu$ in $L^2(0,T;W^{1,2}(\mt))$, c.f. Proposition \ref{prop:skorokhod1}.
Combining \eqref{conv:Prhou} with \eqref{conv:Pu} we conclude that
\begin{align*}%\label{conv:newu}
\mathcal P_H(\tilde\varrho_\varepsilon\tilde{\bfu}_\varepsilon)\cdot\mathcal P_H\tilde\bfu_\varepsilon &\rightharpoonup|\tilde\bfu|^2\quad\text{in}\quad L^1(Q)\quad\tilde\prst\text{-a.s}.
\end{align*}
Using Proposition \ref{prop:skorokhod1} yields $\tilde{\p}$-a.s.
\begin{align*}
\Big|\int_Q \big(|\mathcal P_H\tilde{\bfu}_\varepsilon|^2-\mathcal P_H(\tilde\varrho_\varepsilon\tilde{\bfu}_\varepsilon)\cdot\mathcal P_H\tilde\bfu_\varepsilon\big)\dxt\Big|&\leq \|\tilde\varrho_\varepsilon-1\|_{L^\infty(0,T;L^\gamma)}\|\tilde\bfu_\varepsilon\|^2_{L^2(0,T;L^s)}\\
&\longrightarrow 0,
\end{align*}
where $s=\frac{2\gamma}{\gamma-1}<\frac{2N}{N-2}$. This implies $\|\mathcal P_H\tilde\bfu_\varepsilon\|_2\rightarrow \|\tilde\bfu\|_2$ and hence
\begin{align*}%\label{conv:newu}
\mathcal P_H\tilde\bfu_\varepsilon &\rightarrow\tilde\bfu\quad\text{in}\quad L^2(0,T;L^{2}(\mt)).
\end{align*}
Combining this with weak convergence in $L^2(0,T;W^{1,2}(\mt))$ (recall Proposition \ref{prop:skorokhod1}) yields
 the claim.
\end{proof}

\iffalse
\begin{lemma}\label{lem:strongq}
We have for all $q<\tfrac{2\gamma}{\gamma+1}$
$$\tilde\varrho_\varepsilon\tilde{\bfu}_\varepsilon\rightarrow\tilde{\bfu}\qquad\text{in}\qquad L^q(\tilde\Omega\times Q).$$
\end{lemma}
\begin{proof}
According to \eqref{aprhov}, \eqref{aprho}, \eqref{conv:rho2} and Proposition \ref{prop:skorokhod1}, \eqref{apv}, it follows (up to a subsequence) that
$$\sqrt{\tilde\varrho_\varepsilon}\tilde\bfu_\varepsilon\rightharpoonup \tilde\bfu\quad\text{in} \quad L^2(\tilde\Omega\times Q).$$
Besides, similar to the proof of \eqref{conv:newu} we have
\begin{align*}
\tilde\E\int_Q \tilde\varrho_\varepsilon|\tilde\bfu_\varepsilon|^2\dxt=\tilde\E\int_Q \tilde\varrho_\varepsilon\tilde\bfu_\varepsilon\cdot\tilde\bfu_\varepsilon\dxt\rightarrow \tilde\E\int_Q \tilde\bfu\cdot\tilde\bfu\dxt,
\end{align*}
which is the convergence of the corresponding $L^2$-norms. Thus
$$\sqrt{\tilde\varrho_\varepsilon}\tilde\bfu_\varepsilon\rightarrow \tilde\bfu\quad\text{in}\quad L^2(\tilde\Omega\times Q).$$
Combining this with \eqref{conv:rho2} and \eqref{estrhou2} implies the claim.
\end{proof}

\fi

In the following we aim to identify the limit in the gradient part of the convective term. To this end, we adopt the deterministic approach proposed by Lions and Masmoudi \cite{LiMa2}. We introduce the dual space
\[
W_{\divergence}^{-l,2}(\mt) \equiv \left[ W_{\divergence}^{l,2}(\mt) \right]^*.
\]
In particular, two elements of $W_{\divergence}^{-l,2}(\mt)$ are identical if their difference is a gradient.

\iffalse
We write
\begin{align}\label{eq:24}
\begin{pmatrix} \tilde\varphi_\varepsilon \\ \mathcal Q(\tilde\varrho_\varepsilon\tilde\bfu_\varepsilon) \end{pmatrix}=\mathcal S\Big(\frac{t}{\varepsilon}\Big)\begin{pmatrix} \tilde\psi \\ \tilde\bfm \end{pmatrix}+\tilde\bfr_\varepsilon,\quad \tilde\bfr_\varepsilon =\mathcal S\Big(\frac{t}{\varepsilon}\Big)\begin{pmatrix} \tilde\psi_\varepsilon-\tilde\psi \\ \tilde\bfm_\varepsilon-\tilde\bfm \end{pmatrix}.
\end{align}
Using again that $\mathcal S$ is an isometry we get
\begin{align}\label{eq:convre}
\tilde\bfr_\varepsilon\rightarrow0\quad \text{in}\quad L^2(0,T;W^{-1,2}(\mt))\quad\tilde\p\text{-a.s.}
\end{align}
as a consequence of Proposition \ref{prop:skorokhod1}.
We introduce the function
\begin{align*}
 \tilde\bfv_\varepsilon =\mathcal S_2\Big(\frac{t}{\varepsilon}\Big)\begin{pmatrix} \tilde\psi \\ \tilde\bfm \end{pmatrix}
\end{align*}
which turns out to be very important for the following limit.
\fi

\begin{proposition}\label{conv:convect}
For $l>\tfrac{N}{2}$ we have $\tilde\p$-a.s.
\begin{align*}
\divergence(\tilde\varrho_\varepsilon\tilde\bfu_\varepsilon\otimes\tilde\bfu_\varepsilon)
\rightharpoonup \divergence(\tilde\bfu\otimes\tilde\bfu)\quad\text{in}\quad L^1(0,T;W_{\divergence}^{-l,2}(\mt)).
\end{align*}
\end{proposition}

\begin{proof}
Following \cite{LiMa2} we decompose
\begin{align*}
\tilde\varrho_\varepsilon\tilde\bfu_\varepsilon&=\tilde\bfu+\mathcal P_H\big(\tilde\varrho_\varepsilon\tilde\bfu_\varepsilon-\tilde\bfu\big)+\mathcal Q\big(\tilde\varrho_\varepsilon\tilde\bfu_\varepsilon-\tilde\bfu\big),\\
\tilde\bfu_\varepsilon&=\tilde\bfu+\mathcal P_H\big(\tilde\bfu_\varepsilon-\tilde\bfu\big)+\mathcal Q\big(\tilde\bfu_\varepsilon-\tilde\bfu\big).
\end{align*}
The claim follows once we can show that the following convergences hold true weakly in $L^1(0,T;W_{\divergence}^{-l,2}(\mt))$ $\tilde\p$-a.s.:
\begin{align}
&\divergence\Big(\tilde\bfu\otimes\mathcal P_H\big(\tilde\bfu_\varepsilon-\tilde\bfu\big) \Big) \rightharpoonup 0,\label{eq:convQ1}\\
&\divergence\Big(\tilde\bfu\otimes\mathcal Q\big(\tilde\bfu_\varepsilon-\tilde\bfu\big) \Big)
\rightharpoonup 0,\label{eq:convQ2}\\
&\divergence\Big(\mathcal P_H\big(\tilde\varrho_\varepsilon\tilde\bfu_\varepsilon-\tilde\bfu\big)\otimes\tilde\bfu\Big)\rightharpoonup 0,\label{eq:convQ3}\\
&\divergence\Big(\mathcal Q\big(\tilde\varrho_\varepsilon\tilde\bfu_\varepsilon-\tilde\bfu\big)\otimes\tilde\bfu\Big)\rightharpoonup 0,\label{eq:convQ4}\\
&\divergence\Big( P_H\big(\tilde\varrho_\varepsilon\tilde\bfu_\varepsilon-\tilde\bfu\big)\otimes\mathcal P_H\big(\tilde\bfu_\varepsilon-\tilde\bfu\big)\Big)\rightharpoonup 0,\label{eq:convQ5}\\
&\divergence\Big( P_H\big(\tilde\varrho_\varepsilon\tilde\bfu_\varepsilon-\tilde\bfu\big)\otimes\mathcal Q\big(\tilde\bfu_\varepsilon-\tilde\bfu\big)\Big)\rightharpoonup 0,\label{eq:convQ6}\\
&\divergence\Big(\mathcal Q\big(\tilde\varrho_\varepsilon\tilde\bfu_\varepsilon-\tilde\bfu\big)\otimes\mathcal P_H\big(\tilde\bfu_\varepsilon-\tilde\bfu\big)\Big)\rightharpoonup 0,\label{eq:convQ7}\\
&\divergence\Big(\mathcal Q\big(\tilde\varrho_\varepsilon\tilde\bfu_\varepsilon-\tilde\bfu\big)\otimes\mathcal Q\big(\tilde\bfu_\varepsilon-\tilde\bfu\big)\Big)\rightharpoonup 0,\label{eq:convQ8}
\end{align}
The first four convergences follow from Proposition \ref{prop:skorokhod1}, \eqref{conv:rhou} and the continuity of $\mathcal P_H$ and $\mathcal Q$ respectively. The convergences
\eqref{eq:convQ5}-\eqref{eq:convQ7} are consequences of \eqref{conv:rho2} and \eqref{conv:newu}. In fact, the only critical part is \eqref{eq:convQ8}. First, we need some improved space regularity.
Similarly to \cite{LiMa2}, we use
mollification by means of spatial convolution with a family of regularizing kernels with a parameter $0<\kappa\ll1$. As a matter of fact, thanks to the special geometry of the flat torus $\mt$, the mollified functions can be taken as projections to a finite number (which is the smallest natural number $\geq\frac{1}{\kappa}$) of
modes of the trigonometric basis $\{ \exp(ikx) \}_{k \in Z}$. In particular, the mollification commutes with all spatial derivatives as well as with
the projections $\mathcal P_h$ and $\mathcal Q$.
For $\delta>0$ arbitrary we take $\kappa=\kappa(\delta)$ so small that
\begin{align}\label{eq:regx1}
\tilde\E\|(\tilde\varrho_\varepsilon\tilde\bfu_\varepsilon)^\kappa-\tilde\varrho_\varepsilon^\kappa\tilde\bfu^\kappa_\varepsilon\|_{L^2(L^{\frac{2\gamma}{\gamma+1}})}+\tilde\E\|(\tilde\varrho_\varepsilon\tilde\bfu_\varepsilon)^\kappa-\tilde\varrho_\varepsilon\tilde\bfu_\varepsilon\|_{L^2(L^{\frac{2\gamma}{\gamma+1}})}\leq \delta,\\
\tilde\E\|(\tilde\varrho_\varepsilon\tilde\bfu_\varepsilon)^\kappa-\tilde\bfu_\varepsilon^\kappa\|_{L^2(L^{\frac{2\gamma}{\gamma+1}})}
+\tilde\E\|\tilde\bfu_\varepsilon^\kappa-\tilde\bfu_\varepsilon\|_{L^2(L^{\frac{2N}{N-2}})}+\tilde\E\|\tilde\bfu^\kappa-\tilde\bfu\|_{L^2(L^{\frac{2N}{N-2}})}\leq \delta,\label{eq:regx2}
\end{align}
uniformly in $\varepsilon$. We note that the norm $\tilde\E\|\tilde\bfu_\varepsilon^\kappa-\tilde\bfu_\varepsilon\|_{L^2(L^{\frac{2N}{N-2}})}$
can be made uniformly small as a consequence of the gradient estimate (\ref{apv-}).
As the mollification commutes with $\divergence$ and $\mathcal Q$, it is enough to show that $\tilde\p$-a.s.
\begin{align}
&\divergence\Big(\mathcal Q\big(\tilde\varrho_\varepsilon^\kappa\tilde\bfu_\varepsilon^\kappa-\tilde\bfu^\kappa\big)\otimes\mathcal Q\big(\tilde\bfu_\varepsilon^\kappa-\tilde\bfu^\kappa\big)\Big)\rightharpoonup 0,\label{eq:convQ8'}
\end{align}
for fixed $\kappa$ instead of \eqref{eq:convQ8} (in fact expectation of the $L^1(0,T;W_{\divergence}^{-l,2}(\mt))$-norm of the difference of \eqref{eq:convQ8'} and \eqref{eq:convQ8} can be estimated in terms of $\delta$ using \eqref{eq:regx1} and \eqref{eq:regx2}).
To prove \eqref{eq:convQ8'} we write
\begin{align*}
\mathcal Q\big(\tilde\bfu_\varepsilon^\kappa-\tilde\bfu^\kappa\big)=\mathcal Q\big(\tilde\varrho_\varepsilon^\kappa\tilde\bfu^\kappa_\varepsilon-\tilde\bfu^\kappa\big)+\mathcal Q\big((1-\tilde\varrho^\kappa_\varepsilon)\tilde\bfu^\kappa_\varepsilon\big).
\end{align*}
By \eqref{conv:rho2}, the continuity of $\mathcal Q$ and the boundedness of $\tilde\bfu_\varepsilon^\kappa$ we know that
\begin{align*}
\mathcal Q\big((1-\tilde\varrho^\kappa_\varepsilon)\tilde\bfu_\varepsilon^\kappa\big)\rightarrow0\quad\text{in}\quad L^2(Q)
\end{align*}
$\tilde\p$-a.s. So  \eqref{eq:convQ8'} follows from
\begin{align}
\mathcal \divergence\Big(Q\big(\tilde\varrho^\kappa_\varepsilon\tilde\bfu^\kappa_\varepsilon-\tilde\bfu^\kappa\big)\otimes\mathcal Q\big(\tilde\varrho^\kappa_\varepsilon\tilde\bfu^\kappa_\varepsilon-\tilde\bfu^\kappa\big) \Big)\rightharpoonup 0,\label{eq:convQ9}
\end{align}
in $L^1(0,T;W_{\divergence}^{-l,2}(\mt))$.
As $\divergence\big(\mathcal Q\tilde\bfu^\kappa\otimes\mathcal Q\tilde\bfu^\kappa\big)=\tfrac{1}{2}\nabla |\mathcal Q\tilde\bfu^\kappa|^2$,
the convergence \eqref{eq:convQ9} is a consequence of
 \begin{align}
\divergence\Big(\mathcal Q\big(\tilde\varrho_\varepsilon\tilde\bfu_\varepsilon\big)^\kappa\otimes\mathcal Q\big(\tilde\varrho_\varepsilon\tilde\bfu_\varepsilon\big)^\kappa\Big)\rightharpoonup 0\quad\text{in}\quad L^1(0,T;W_{\divergence}^{-l,2}(\mt)),\label{eq:convQ10}
\end{align}
thanks to \eqref{conv:rhou} and \eqref{eq:regx1}.
In order to show \eqref{eq:convQ10} (we need to introduce the function $\tilde\Psi_\varepsilon=\Delta^{-1}\divergence(\tilde\varrho_\varepsilon\tilde\bfu_\varepsilon)$
which satisfies $\nabla\tilde\Psi_\varepsilon=\mathcal Q(\tilde\varrho_\varepsilon\tilde\bfu_\varepsilon)$. We have the system of equations
\begin{align*}%\label{eq:Psiphi}
\dd(\varepsilon\tilde\varphi_\varepsilon)=-\nabla\tilde\Psi_\varepsilon\dt,\quad \dd\nabla\tilde\Psi_\varepsilon=-\frac{\gamma}{\varepsilon}\nabla\tilde\varphi_\varepsilon\dt+\tilde\bfF_\varepsilon\dt+\mathcal Q\Phi(\tilde\varrho_\varepsilon,\tilde\varrho_\varepsilon\tilde\bfu_\varepsilon)\dd \tilde W_\varepsilon.
\end{align*}
The right-hand-side only belongs to $W^{-l,2}(\mt)$. So
we apply mollification and gain
$\tilde\Psi_\varepsilon^\kappa=\Delta^{-1}\divergence(\tilde\varrho_\varepsilon\tilde\bfu_\varepsilon)^\kappa$
and $\nabla\tilde\Psi_\varepsilon^\kappa=\mathcal Q(\tilde\varrho_\varepsilon\tilde\bfu_\varepsilon)^\kappa$. The system of equations
for $\tilde\varphi^\kappa_\varepsilon$ and $\tilde\Psi_\varepsilon^\kappa$ reads as
\begin{align}\label{eq:Psiphi}
\dd(\varepsilon\tilde\varphi_\varepsilon^\kappa)=-\Delta\tilde\Psi^\kappa_\varepsilon\dt,\quad \dd\nabla\tilde\Psi_\varepsilon^\kappa=-\frac{\gamma}{\varepsilon}\nabla\tilde\varphi_\varepsilon^\kappa\dt+\tilde\bfF^\kappa_\varepsilon\dt+\mathcal Q\Phi(\tilde\varrho_\varepsilon,\tilde\varrho_\varepsilon\tilde\bfu_\varepsilon)^\kappa\dd \tilde W_\varepsilon.
\end{align}
We note that for the special choice, where the mollification is taken as the projection onto a finite number of Fourier modes, the
system (\ref{eq:Psiphi}) reduces to a \emph{finite number} of equations.
Now, we apply It\^{o}'s formula to the function $$f(\varepsilon\tilde\varphi_\varepsilon^\kappa,\nabla\tilde\Psi_\varepsilon^\kappa)=\int_{\mt}\varepsilon\tilde\varphi^\kappa_\varepsilon\nabla\tilde\Psi^\kappa_\varepsilon\cdot\bfvarphi\dx,$$
with $\bfvarphi\in C^\infty_{\divergence}(\mt)$ arbitrary
and gain
\begin{align*}
\int_{\mt}&\varepsilon\tilde\varphi_\varepsilon^\kappa(t)\nabla\tilde\Psi^\kappa_\varepsilon(t)\cdot\bfphi\dx\\
&=-\int_0^t\int_{\mt}\Delta\tilde\Psi^\kappa_\varepsilon\nabla\tilde\Psi^\kappa_\varepsilon\cdot\bfphi\dxs
-\gamma\int_0^t\int_{\mt}\tilde\varphi^\kappa_\varepsilon\nabla\tilde\varphi^\kappa_\varepsilon\cdot\bfphi\dxs\\
&+\varepsilon\int_0^t\int_{\mt}\tilde\varphi_\varepsilon^\kappa\tilde\bfF^\kappa_\varepsilon\cdot\bfphi\dxs
+\varepsilon\int_{\mt}\int_0^t\tilde\varphi_\varepsilon^\kappa\bfphi\cdot\mathcal Q\Phi(\tilde\varrho_\varepsilon,\tilde\varrho_\varepsilon\tilde\bfu_\varepsilon)^\kappa\dd \tilde W_\varepsilon\dx.
\end{align*}
And we have
\begin{align*}
\int_0^t\int_{\mt}&\Delta\tilde\Psi^\kappa_\varepsilon\nabla\tilde\Psi^\kappa_\varepsilon\cdot\bfphi\dxs
\\&=\frac{1}{2}\int_0^t\int_{\mt}\nabla|\nabla\tilde\Psi^\kappa_\varepsilon|^2\cdot\bfphi\dxs-\int_0^t\int_{\mt}\nabla\tilde\Psi^\kappa_\varepsilon\otimes\nabla\tilde\Psi_\varepsilon^\kappa:\nabla\bfphi\dxs\\
&=-\int_0^t\int_{\mt}\nabla\tilde\Psi^\kappa_\varepsilon\nabla\tilde\Psi^\kappa_\varepsilon:\nabla\bfphi\dxs,\\
\int_0^t\int_{\mt}&\tilde\varphi_\varepsilon^\kappa\nabla\tilde\varphi^\kappa_\varepsilon\cdot\bfphi\dxs
=\frac{1}{2}\int_0^t\int_{\mt}\nabla|\tilde\varphi^\kappa_\varepsilon|^2\cdot\bfphi\dxs=0,
\end{align*}
due to $\divergence\bfphi=0$. So we end up with
\begin{align*}
\int_0^t\int_{\mt}&\nabla\tilde\Psi^\kappa_\varepsilon\otimes\nabla\tilde\Psi^\kappa_\varepsilon:\nabla\bfphi\dxs=-\varepsilon\int_{\mt}\tilde\varphi^\kappa_\varepsilon(t)\nabla\tilde\Psi^\kappa_\varepsilon(t)\cdot\bfphi\dx\\
&+\varepsilon\int_0^t\int_{\mt}\tilde\varphi_\varepsilon^\kappa\tilde\bfF^\kappa_\varepsilon\cdot\bfphi\dxs
+\varepsilon\int_{\mt}\int_0^t\tilde\varphi_\varepsilon^\kappa\bfphi\cdot\mathcal Q\Phi(\tilde\varrho_\varepsilon,\tilde\varrho_\varepsilon\tilde\bfu_\varepsilon)^\kappa\dd \tilde W_\varepsilon\dx.
\end{align*}
For fixed $\kappa>0$ the right-hand-side vanishes $\tilde\p$-a.s. for $\varepsilon\rightarrow0$ at least after taking a subsequence due to Corollary \ref{cor:limit1}, Proposition \ref{prop:skorokhod1} and the properties of the mollification.
Finally we conclude with \eqref{eq:convQ10} which implies the last missing convergence \eqref{eq:convQ8} as explained above.
\end{proof}

Now, we have all in hand to complete the proof of Theorem \ref{thm:1} which implies the proof of our main result, Theorem \ref{thm:main}.

%\begin{theorem}
%The process $\tilde W$ is a $(\tilde\mf_t)$-cylindrical Wiener process and
%$$\big((\tilde\Omega,\tilde\mf,(\tilde\mf_t),\tilde\prst),\tilde \bfu,\tilde W\big)$$
%is a weak martingale solution to \eqref{eq:lim} with the initial law $\Gamma$.
%
%\end{theorem}

\begin{proof}[Proof of Theorem \ref{thm:1}]
The first part of the claim follows immediately from the fact that all $\tilde W_\varepsilon$ are cylindrical Wiener processes due to Proposition \ref{prop:limit1}. As a consequence, there exists a collection of mutually independent real-valued $(\tilde{\mf}_t)$-Wiener processes $(\tilde{\beta}_k)_{k\geq1}$ such that $\tilde{W}=\sum_{k\geq1}\tilde{\beta}_k e_k$.

In order to show that \eqref{eq:lim} is satisfied in the sense of Definition \ref{def:inc}, let us take a divergence free test function $\bfphi\in C^\infty_{\text{div}}(\mt)$ and consider the functionals $M,\,N,\,N_k$ from Proposition \ref{prop:limit1}. This way we only study the approximate equation \eqref{eq2} projected by $\mathcal{P}_H$ and the pressure term drops out. Having \eqref{exp11}, \eqref{exp21} and \eqref{exp31} in hand, we intend to pass to the limit as $\varepsilon\rightarrow 0$ and to deduce
\begin{equation}\label{exp111}
\begin{split}
&\tilde{\stred}\,h\big(\bfr_s\tilde{\bu},\bfr_s\tilde{W}\big)\big[M(1,\tilde\bu,\tilde\bfu)_{s,t}\big]=0,
\end{split}
\end{equation}
\begin{equation}\label{exp211}
\begin{split}
&\tilde{\stred}\,h\big(\bfr_s\tilde{\bu},\bfr_s\tilde{W}\big)\bigg[[M(1,\tilde\bu,\tilde\bfu)^2]_{s,t}-N(1,\tilde\bfu)_{s,t}\bigg]=0,
\end{split}
\end{equation}
\begin{equation}\label{exp311}
\begin{split}
&\tilde{\stred}\,h\big( \bfr_s\tilde{\bu},\bfr_s\tilde{W}\big)\bigg[[M(1,\tilde\bu,\tilde\bfu)\tilde{\beta}_k]_{s,t}-N_k(1,\tilde\bfu)_{s,t}\bigg]=0.
\end{split}
\end{equation}
Note that the proof will then be complete. Indeed, \eqref{exp111}, \eqref{exp211} and \eqref{exp311} imply that the process $M(1,\tilde\bfu,\tilde\bfu)$ is a $(\tilde\mf_t)$-martingale and its quadratic and cross variations satisfy, respectively,
\begin{equation*}%\label{mart}
\begin{split}
\langle\!\langle M(1,\tilde\bfu,\tilde\bfu)\rangle\!\rangle&=N(1,\tilde\bfu),\quad\qquad\langle\!\langle M(1,\tilde\bfu,\tilde\bfu),\tilde \beta_k\rangle\!\rangle=N_k(1,\tilde\bfu),
\end{split}
\end{equation*}
and consequently
$$\bigg\langle\!\!\!\bigg\langle M(1,\tilde\bu,\tilde\bfu)-\int_0^\tec \big\langle\varPhi(1,\tilde\bu)\,\dif \tilde W,\bfvarphi\big\rangle\bigg\rangle\!\!\!\bigg\rangle=0$$
hence \eqref{eq2lim} is satisfied in the sense required by Definition \ref{def:inc}.

Let us now verify \eqref{exp111}, \eqref{exp211} and \eqref{exp311}. First of all we observe that
$$M(\tilde\varrho_\varepsilon,\tilde\bfu_\varepsilon,\tilde\varrho_\varepsilon\tilde\bfu_\varepsilon)_t\rightarrow M(1,\tilde\bfu,\tilde\bfu)_t\quad\text{a.s.}$$
due to Proposition \ref{prop:skorokhod1}, Proposition \ref{conv:convect} and \eqref{conv:Prhou}. Application of the Vitali convergence theorem together with the uniform estimates \eqref{apv}, \eqref{estrhou2} and \eqref{estrhouu} justifies the passage to the limit in \eqref{exp11} and \eqref{exp111} follows. The same argument implies the passage to the limit in the part of \eqref{exp21} and \eqref{exp31} involving $M$.

Finally, we comment on the passage to the limit in the terms coming from the stochastic integral, i.e. $N$ and $N_k$. The convergence in \eqref{exp31} being easier, let us only focus on \eqref{exp21} in detail.
As the first step we note that the convergence
\begin{align*}
\sum_{k\geq1}\big\langle \bfg_k(\tilde\varrho_\varepsilon,\tilde\varrho_\varepsilon\tilde\bfu_\varepsilon),\bfvarphi\big\rangle^2\rightarrow \sum_{k\geq 1}\big\langle \bfg_k(1,\tilde\bfu),\bfvarphi\big\rangle^2\qquad\tilde\prst\otimes\mathcal{L}\text{-a.e}.
\end{align*}
follows once we show that
\begin{equation}\label{convL2}
\big\langle\varPhi(\tilde\varrho_\varepsilon,\tilde\varrho_\varepsilon\tilde\bfu_\varepsilon)\,\cdot\,,\bfvarphi\big\rangle\rightarrow\big\langle\varPhi(1,\tilde\bfu)\,\cdot\,,\bfvarphi\big\rangle\qquad\text{in}\qquad L_2(\mathfrak{U};\mr)\qquad\tilde\prst\otimes\mathcal{L}\text{-a.e.}
\end{equation}
To this end, we write
\begin{align*}
&\big\|\big\langle\varPhi(\tilde\varrho_\varepsilon,\tilde\varrho_\varepsilon\tilde\bfu_\varepsilon)\,\cdot\,,\bfvarphi\big\rangle-\big\langle\varPhi(1,\tilde\bfu)\,\cdot\,,\bfvarphi\big\rangle\big\|_{L_2(\mathfrak{U};\mr)}\\
&\quad\leq\bigg(\sum_{k\geq 1}\big|\big\langle \bfh_k(\tilde\varrho_\varepsilon)-\bfh_k(1),\bfvarphi\big\rangle\big|^2\bigg)^\frac{1}{2}+\bigg(\sum_{k\geq 1}|\alpha_k|^2\big|\big\langle \tilde\varrho_\varepsilon\tilde\bfu_\varepsilon-\tilde\bfu,\bfvarphi\big\rangle\big|^2\bigg)^\frac{1}{2}\\
&\quad =I_1+I_2.
\end{align*}
For $I_2$ we use \eqref{growth1-} together with \eqref{conv:Prhou} to obtain
$I_2\rightarrow0$ for a.e. $(\omega,t)$.
For $I_1$ we apply the Minkowski integral inequality, the mean value theorem, \eqref{growth1} and \eqref{growth2} to obtain
\begin{align*}
I_1&\leq C\bigg(\,\sum_{k\geq1}\big\| \bfh_k(\tilde \varrho_\varepsilon)-\bfh_k(1)\big\|_{L^1_x}^2\bigg)^{\frac{1}{2}}\leq C\int_{\mt}\!\bigg(\sum_{k\geq1}\big|\bfh_k(\tilde \varrho_\varepsilon)-\bfh_k(1)\big|^2\bigg)^{\frac{1}{2}}\dif x\\
&\leq C\int_{\mt}\Big(1+\tilde\varrho_\varepsilon^{\frac{\gamma-1}{2}}\Big)|\tilde\varrho_\varepsilon-1|\,\dif x\leq C\bigg[\int_{\mt}\Big(1+\tilde\varrho_\varepsilon^{\frac{\gamma-1}{2}}\Big)^{p}\,\dif x\bigg]^{\frac{1}{{p}}}\bigg[\int_{\mt}|\tilde\varrho_\varepsilon-1|^{q}\,\dif x\bigg]^{\frac{1}{q}}
\end{align*}
where the conjugate exponents $p,q\in(1,\infty)$ are chosen in such a way that
$$p\frac{\gamma-1}{2}<\gamma+1\qquad\text{and}\qquad q<\gamma.$$
Therefore, using \eqref{aprho}, \eqref{conv:rho2} we deduce
\begin{align*}
\tilde\E \int_0^TI_1\dt\rightarrow 0.
\end{align*}
and so for a subsequence $I\rightarrow 0$ for a.e. $(\omega, t)$ and \eqref{convL2} follows. Besides, since, for all $p\geq 2$,
\begin{align*}
\tilde\stred\int_s^t&\big\|\big\langle\varPhi(\tilde \varrho_\varepsilon,\tilde\varrho_\varepsilon\tilde\bfu_\varepsilon)\,\cdot,\bfvarphi\big\rangle\big\|_{L_2(\mathfrak{U};\mr)}^p\,\dif r\\
&\leq C\,\tilde\stred\int_s^t\|\tilde\varrho_\varepsilon\|_{L^2}^{\frac{p}{2}}\Big(1+\|\tilde\varrho_\varepsilon\|_{L^\gamma}^\gamma+\|\sqrt{\tilde\varrho_\varepsilon}\tilde\bfu_\varepsilon\|^2_{L^2}\Big)^{\frac{p}{2}}\dif r\\
&\leq C\bigg(1+\tilde\stred\sup_{0\leq t\leq T}\|\tilde\varrho_\varepsilon\|_{L^\gamma}^{\gamma p}+\tilde\stred\sup_{0\leq t\leq T}\|\sqrt{\tilde\varrho_\varepsilon}\tilde\bfu_\varepsilon\|_{L^{2}}^{2p}\bigg)\leq C
\end{align*}
due to \eqref{aprhov}, \eqref{aprho}, we obtain the convergence in \eqref{exp21} and therefore $\tilde\bfu$ solves \eqref{eq:lim}.
It follows immediately from our construction that for all $p\in[1,\infty)$
$$\tilde\bfu\in L^p(\tilde\Omega;L^2(0,T;W^{1,2}_{\text{div}}(\mt))).$$
Besides, since we have (due Proposition \ref{prop:skorokhod1} and \eqref{conv:rho2})
\begin{align*}
\sqrt{\tilde\varrho_\varepsilon}\tilde\bfu_\varepsilon\rightharpoonup \tilde\bfu\quad\text{in}\quad L^1(\Omega;L^1(Q))
\end{align*}
lower semi-continuity of the functional
\begin{align*}
\tilde\bfw\mapsto \tilde\E\bigg[\sup_{t\in(0,T)}\int_{\mt}|\tilde\bfw|^2\dx\bigg]^{\frac{p}{2}}
\end{align*}
yields $\tilde\bfu\in L^p(\tilde\Omega;L^\infty(0,T;L^2(\mt)))$ on account of Corollary \ref{cor:limit1}. The usual argument about the fractional time derivative (in the distributional sense) implies
$$\tilde\bfu\in L^p(\tilde\Omega;C_w([0,T];L_{\text{div}}^{2}(\mt)))$$
%$\tilde\bfu $ is also the limit of $\mathcal P_H(\tilde\varrho_\varepsilon\tilde\bfu_\varepsilon)$ in $\mathcal{X}_{\varrho\bfu}$, it follows from \eqref{estrhou2} that
%$$\tilde\bfu\in L^p(\tilde\Omega;C_w([0,T];L_{\text{div}}^{\frac{2\gamma}{\gamma+1}}(\mt)))$$
%\rmk{was this problem with energy inequality only fro the compressible system? I mean, can we now test by $\tilde\bfu$ (via the ito formula) to get the estimate in
%$L^\infty(0,T;L^2)$ and to conclude the weak continuity there?}\textcolor{red}{We can not test with $\tilde\bfu$. The bound in $L^\infty(0,T;L^2)$ follows from lower-semicontinuity. Weak continuity is first given in some negative space, combining the with $L^\infty(0,T;L^2)$ implies $C_w([0,T];L^2)$}\rmk{exactly my point. but how exactly do we get $L^\infty(L^2)$ for $\tilde\bfu$? the lower semicontinuity applied to $\varrho_\varepsilon\bfu_\varepsilon$ only gives $L^\infty(L^\frac{2\gamma}{\gamma+1})$, no?}
and the proof is complete.
\end{proof}

%As a consequence, we deduce that the result of Theorem \ref{thm:main} is proven.

\section{Proof of Theorem \ref{thm:main2d}}
\label{subsec:strong}

In order to complete the proof of Theorem \ref{thm:main2d}, we make use of Proposition \ref{diagonal} which is a generalization of the Gy\"{o}ngy-Krylov characterization of convergence in probability introduced in \cite{krylov} adapted to the case of quasi-Polish spaces. It applies to situations when pathwise uniqueness and existence of a martingale solution are valid and allows to establish existence of a pathwise solution. We recall that in the case of $N=2$ pathwise uniqueness for \eqref{eq:lim} is known (cf. Theorem \ref{thm:inc}).

We consider the collection of joint laws of
$$(\varrho_n,\bfu_n,\mathcal{P}(\varrho_n\bfu_n),\varrho_m,\bfu_m,\mathcal{P}(\varrho_m\bfu_m))\quad\text{on}\quad\mathcal{X}_\varrho\times\mathcal{X}_{\bfu}\times\mathcal{X}_{\varrho\bfu}\times\mathcal{X}_\varrho\times\mathcal{X}_{\bfu}\times\mathcal{X}_{\varrho\bfu},$$
denoted by $\mu^{n,m}$. For this purpose we define the extended path space
$$\mathcal{X}^J=\mathcal{X}_\varrho\times\mathcal{X}_{\bfu}\times\mathcal{X}_{\varrho\bfu}\times\mathcal{X}_\varrho\times\mathcal{X}_{\bfu}\times\mathcal{X}_{\varrho\bfu}\times\mathcal{X}_W$$
As above, denote by $\mu_W$ the law of $W$ and set $\nu^{n,m}$ to be the joint law of
$$(\varrho_n,\bfu_n,\mathcal{P}(\varrho_n\bfu_n),\varrho_m,\bfu_m,\mathcal{P}(\varrho_m\bfu_m),W)\quad\text{on}\quad\mathcal{X}^J.$$
Similarly to Corollary \ref{cor:tight} the following fact holds true. The proof is nearly identical and so will be left to the reader.

\begin{proposition}
The collection $\{\nu^{n,m};\,n,m\in\mn\}$ is tight on $\mathcal{X}^J$.
\end{proposition}

Let us take any subsequence $\{\nu^{n_k,m_k};\,k\in\mn\}$. By the Jakubowski-Skorokhod theorem, Theorem \ref{thm:jakubow}, we infer (for a further subsequence but without loss of generality we keep the same notation) the existence a probability space $(\bar{\Omega},\bar{\mf},\bar{\prst})$ with a sequence of random variables
$$(\hat\varrho_{n_k},\hat\bfu_{n_k},\hat\bfq_{n_k},\check\varrho_{m_k},\check\bfu_{m_k},\check\bfq_{m_k},\bar W_k),\quad k\in\mn,$$
conver\-ging almost surely in $\mathcal{X}^J$ to a random variable
$$(\hat\varrho,\hat\bfu,\hat\bfq,\check\varrho,\check\bfu,\check\bfq,\bar W)$$
and
$$\bar{\prst}\big((\hat\varrho_{n_k},\hat\bfu_{n_k},\hat\bfq_{n_k},\check\varrho_{m_k},\check\bfu_{m_k},\check\bfq_{m_k},\bar W_k)\in\,\,\cdotp\big)=\nu^{n_k,m_k}(\cdot).$$
%$$\bar{\prst}\big((\hat\varrho,\hat\bfu,\hat\bfq,(\hat\psi,\hat\bfm),\check\varrho,\check\bfu,\check\bfq,(\check\psi,\check\bfm),\bar W)\in\,\,\cdotp\big)=\nu(\cdot).$$
Observe that in particular, $\mu^{n_k,m_k}$ converges weakly to a measure $\mu$ defined by
$$\mu(\cdot)=\bar{\prst}\big((\hat\varrho,\hat\bfu,\hat\bfq,\check\varrho,\check\bfu,\check\bfq)\in\,\,\cdotp\big).$$
As the next step, we should recall the technique established in Subsection \ref{subsec:ident}. Analogously, it can be applied to both
$$(\hat\varrho_{n_k},\hat\bfu_{n_k},\hat\bfq_{n_k},\bar W_k),\;(\hat\varrho,\hat\bfu,\hat\bfq,\bar W)$$
and
$$(\check\varrho_{m_k},\check\bfu_{m_k},\check\bfq_{m_k},\bar W_k),\;(\check\varrho,\check\bfu,\check\bfq,\bar W)$$
in order to show that $(\hat\bfu,\bar W)$ and $(\check\bfu,\bar W)$ are weak martingale solutions to \eqref{eq:lim} defined on the same stochastic basis $(\bar{\Omega},\bar{\mf},(\bar{\mf}_t),\bar{\prst})$, where $(\bar{\mf}_t)$ is the $\bar\prst$-augmented canonical filtration of $(\hat\bfu,\check\bfu,\bar W)$.
Besides, we obtain that
$$\hat\varrho=\check\varrho,\quad\hat\bfq=\hat\bfu,\quad\check\bfq=\check\bfu\qquad\bar\prst\text{-a.s.}$$
In order to verify the condition \eqref{eq:diag} from Proposition \ref{diagonal} we employ the pathwise uniqueness result for \eqref{eq:lim} in two dimensions, cf. Theorem \ref{thm:inc}. Indeed, it follows from our assumptions on the approximate initial laws $\Lambda_\varepsilon$ that $\hat{\bfu}(0)=\check{\bfu}(0)=\bfu_0$ $\bar\p$-a.s., therefore according to Theorem \ref{thm:inc2d} the solutions $\hat\bfu$ and $\check\bfu$ coincide $\bar\p$-a.s. and
\begin{align*}
\mu&\Big((\varrho_1,\bfu_1,\bfq_1,\varrho_2,\bfu_2,\bfq_2);\;(\varrho_1,\bfu_1,\bfq_1)=(\varrho_2,\bfu_2,\bfq_2)\Big)\\
&\quad=\bar{\prst}\Big((\hat\varrho,\hat\bfu,\hat\bfq)=(\check\varrho,\check\bfu,\check\bfq)\Big)=\bar{\prst}(\hat{\bfu}=\check{\bfu})=1.
\end{align*}

Now, we have all in hand to apply Proposition \ref{diagonal}. It implies that the original sequence $(\varrho_\varepsilon,\bfu_\varepsilon,\mathcal P(\varrho_\varepsilon\bfu_\varepsilon))$ defined on the initial probability space $(\Omega,\mf,\prst)$ converges in probability in the topology of $\mathcal{X}_\varrho\times\mathcal{X}_{\bfu}\times\mathcal{X}_{\varrho\bfu}$ to a random variable $(\varrho,\bfu,\bfq)$. Without loss of gene\-rality, we assume that the convergence is almost sure and again by the method from Subsection \ref{subsec:ident} we finally deduce that $\bfu$ is a pathwise weak solution to \eqref{eq:lim}. Actually, identification of the limit is more straightforward here since in this case all the work is done for the initial setting and only one fixed driving Wiener process $W$ is considered. The proof of Theorem \ref{thm:main2d} is complete.

\appendix
\section{Quasi-Polish spaces}
\label{sec:appendix}

The so-called quasi-Polish spaces are topological spaces that are not necessarily metrizable but nevertheless they enjoy several important properties of Polish spaces. Let us recall their definition introduced in \cite{jakubow}.

\begin{definition}\label{def:quasip}
Let $(X,\tau)$ be a topological space such that there exists a countable family
$$\{f_n:X\rightarrow[-1,1];\,n\in\mn\}$$
of continuous functions that separate points of $X$.
\end{definition}

Among the properties of quasi-Polish spaces used in the main body of this paper belongs the following Jakubowski-Skorokhod representation theorem, see \cite[Theorem 2]{jakubow}.

\begin{theorem}\label{thm:jakubow}
Let $(X,\tau)$ be a quasi-Polish space and let $\mathcal{S}$ be the $\sigma$-field generated by $\{f_n;\,n\in\mn\}$. If $\{\mu_n;\,n\in\mn\}$ is a tight sequence of probability measures on $(X,\mathcal{S})$, then there exists a subsequence $(n_k)$, a probability space $(\Omega,\mf,\prst)$ with $X$-valued Borel measurable random variables $\{\xi_k;\,k\in\mn\}$ and $\xi$ such that $\mu_{n_k}$ is
the law of $\xi_k$ and $\xi_k$ converges to $\xi$ in $X$ a.s. Moreover, the law of $\xi$ is a Radon measure.
\end{theorem}

Next, we need to adapt the Gy\"{o}ngy-Krylov characterization of convergence in probability introduced in \cite{krylov} to the setting of quasi-Polish spaces. Recall that the original argument for the case of Polish spaces follows from the following simple observation made in \cite[Lemma 1.1]{krylov}.

\begin{lemma}\label{gyongy}
Let $X$ be a Polish space equipped with the Borel $\sigma$-algebra. A sequence of $X$-valued random variables $\{Y_n;\,n\in\mn\}$ converges in probability if and only if for every subsequence of joint laws, $\{\mu_{n_k,m_k};\,k\in\mn\}$, there exists a further subsequence which converges weakly to a probability measure $\mu$ such that
$$\mu\big((x,y)\in X\times X;\,x=y\big)=1.$$
\end{lemma}

In view of our application in Subsection \ref{subsec:strong}, we are interested in the sufficiency of the above condition.

\begin{proposition}\label{diagonal}
Let $(X,\tau)$ be a quasi-Polish space. Let $\{Y_n;\,n\in\mn\}$ be a sequence of $X$-valued random variables. Assume that for every subsequence of their joint laws $\{\mu_{n_k,m_k};\,k\in\mn\}$ there exists a further subsequence which converges weakly to a probability measure $\mu$ such that
\begin{equation}\label{eq:diag}
\mu\big((x,y)\in X\times X;\,x=y\big)=1.
\end{equation}
Then there exists a subsequence $\{Y_{n_l};\,l\in\mn\}$ which converges a.s.

\begin{proof}
Let $\tilde f$ be the one-to-one and continuous mapping defined by
\begin{align*}
\tilde f:X&\rightarrow [-1,1]^{\mn}\\
x&\mapsto \{f_n(x);\,n\in\mn\},
\end{align*}
where $f_n$ were given by Definition \ref{def:quasip}. Since due to assumption
$$(Y_{n_k},Y_{m_k})\overset{d}{\rightarrow} (Y,Y) \quad\text{in}\quad X\times X$$
for every $(n_k),\,(m_k)$ and some $Y$ with the law $\mu$, we deduce from the continuous mapping theorem that
$$\big(\tilde f(Y_{n_k}),\tilde f(Y_{m_k})\big)\overset{d}{\rightarrow} \big(\tilde f(Y),\tilde f(Y)\big)\quad\text{in}\quad [-1,1]^\mn\times [-1,1]^\mn$$
for every $(n_k),\,(m_k)$. Since $[-1,1]^\mn\times [-1,1]^\mn$ is a Polish space, Lemma \ref{gyongy} applies to the sequence $\{\tilde f(Y_n);\,n\in\mn\}$ and the convergence in probability follows. Consequently, there exists a subsequence $\{\tilde f(Y_{n_l});\,n\in\mn\}$ which converges a.s. and it only remains to prove that $Y_{n_l}$ converges to $Y$ a.s. To this end, we proceed by contradiction: Assume that $Y_{n_l}$ does not converge a.s. to $Y$. Then there exists a set of positive probability $\Omega^*\subset\Omega$ such that for all $\omega\in\Omega^*$ there exists a neighborhood $\mathcal{N}(\omega)$ of $Y(\omega)$ and for every $l_0\in\mn$ there exists $l\geq l_0$ such that $Y_{n_l}(\omega)\notin\mathcal{N}(\omega).$ However, as the sequence $\{f_n;\,n\in\mn\}$ separates points of $X$, there exists $n\in\mn$ such that $f_n(Y_{n_l}(\omega))\neq f_n(Y(\omega))$ and as a consequence there exists a neighborhood $\mathcal{V}(\omega)$ of $\tilde f(Y(\omega))$ such that $\tilde f(Y_{n_l}(\omega))\notin \mathcal{V}(\omega).$ This contradicts the a.s. convergence of $\tilde f(Y_{n_l})$ and completes the proof.
\end{proof}
\end{proposition}

\bigskip

\centerline{\bf Acknowledgement}
\noindent{D.B. was partially supported by Edinburgh Mathematical Society.}\\
{The research of E.F. leading to these results has received funding from the European Research Council under the European Union's Seventh Framework Programme (FP7/2007-2013)/ ERC Grant Agreement 320078. The Institute of Mathematics of the Academy of Sciences of the Czech
Republic is supported by RVO:67985840.}\\
The authors thank the referee for his/her very valuable suggestions.

\end{document}